\UseRawInputEncoding

\documentclass[12pt]{article}
\usepackage{amssymb,latexsym}
\usepackage{pdfsync,authblk}

\usepackage{amsmath,amsthm}
\usepackage[colorlinks,backref]{hyperref}\usepackage{cleveref}
\usepackage{indentfirst}

\renewcommand{\d}{\mathrm{d}}

\usepackage{color}

\theoremstyle{plain}
\newtheorem{thm}{Theorem}[section]

\theoremstyle{definition}

\newtheorem{rem}{Remark}
\numberwithin{equation}{section}
\allowdisplaybreaks

\title{Liouville theorem for $ \phi $-$F$-symphonic map , $ \phi $-$F$-harmonic map  and and $ \phi $-$\Phi_{S, p, \varepsilon}$ harmonic map with free boundary}%\footnote{this work is supported }
 \author{Xiang-Zhi Cao\footnote{aaa7756kijlp@163.com} \thanks{NanJing  xiaozhuang university, 211171, Nanjing, China} } 
\begin{document}
	\maketitle
	\tableofcontents

%	\maketitle
%	\footnotetext[1]{text}

\begin{abstract}
	In this paper, we obtained  Liouville theorem for  $ \phi $-$F$-symphonic map ,  $ \phi $-$F$-harmonic map  and $ \phi $-$\Phi_{S, p, \varepsilon}$ harmonic map with free boundary on metric measure space. 

\end{abstract}

	 \section{Introduction}
	 
	 In \cite{MR2050491}, Xu obtained Liouville type theorems for harmonic maps from upper half Euclidean spaces into Riemannian manifolds with free boundaries. In \cite{MR2766664}, Liu obtained  free boundary value problems for $p$-harmonic  maps from  upper half Euclidean spaces, and obtain some Liouville type theorems.

Let $u:\left(M^m, g\right) \rightarrow\left(N^n, h\right)$ be smooth map. Ara  considered $ F $ harmonic map which is the critical point of  the following functional 	 
	 \begin{equation*}
	 	\begin{split}
	 		E_F(u)=\int_M F(\frac{|du |^2}{2}) dv_g.
	 	\end{split}
	 \end{equation*}
One can refer to (\cite{Mitsunori1999Geometry,ara2001stability,ara2001instability}) for more details. We can consider the following functional on metric measure space $ (M,g,e^{-\phi}dv_g) ,$	 
	  \begin{equation*}
	 	\begin{split}
	 		E_{F,\phi}(u)=\int_M F(\frac{|du |^2}{2})e^{-\phi} dv_g.
	 	\end{split}
	 \end{equation*}
	 Its Euler-Lagrange equation is 
	 \begin{equation}
	 	\delta^\nabla\left( (F^{\prime}\left(\frac{|du|^{2}}{2}\right)du \right)-F^\prime(\frac{|du |^2}{2}) du(\nabla \phi)=0,
	 \end{equation}
	 where $ \sigma_{u}(\cdot)= \operatorname{div}(du) $.  Such kind of map is refer as $ \phi $-$ F $ harmonic map. In Theorem \ref{thm2}, we get Liouville type theorem of free boundary prolem for such kind of map .
	 
	In \cite{kawai2011some} , Kawai and Nakauchi introduced symphoic map which is analogous to harmonic map. Symmonic map is the critical point of 	
	\begin{equation*}
		\begin{split}
			E_{sym}(u)=	\int_M \frac{|u^*h|^2}{2}dv_g.
		\end{split}
	\end{equation*}
Its Euler-Lagrange equation is 	
	
%	On metric measure space $ (M,g,e^{-\phi}dv_g) ,$ we consider the functional 	 
%	 \begin{equation*}
%	 	\begin{split}
%	 		E_{sym}(u)=	\int_M \frac{|u^*h|^2}{2}e^{-\phi} dv_g
%	 	\end{split}
%	 \end{equation*}
%	
 
	\begin{equation}
		\delta^\nabla\left( \sigma_{u} \right)=0,
	\end{equation}
	where  $ \sigma_{u}(\cdot)= \langle du(\cdot) ,du(e_i)\rangle du(e_i) $.
	 
In \cite{MR4338267,MR3451407, Misawa2012}, the authors  studied regularity of symphonic map or m-symphonic map. In \cite{Misawa2023,Misawa2018}, they studied heat flow of symphonic map. In \cite{Misawa2022}, examples of symphonic map are constructed. 
In \cite{MR3940323}, Liouville type results of symphonic map was obtained. 	 

On metric measure space $ (M,g,e^{-\phi}dv_g) ,$ we consider $\phi$-$F $ symphonic map which is the critical point of   the functional 	 
\begin{equation*}
	\begin{split}
		E_{sym}(u)=\int_M F(\frac{|u^*h|^2}{2})e^{-\phi} \d v_g.
	\end{split}
\end{equation*}
Its Euler-Lagrange equation is 
\begin{equation}
	\delta^\nabla\left( (F^{\prime}\left(\frac{|u^*h|^{2}}{2}\right)\sigma_{u} \right)-F^\prime(\frac{|u^*h |^2}{4}) du(\nabla \phi)=0,
\end{equation}
where  $ \sigma_{u}(\cdot)= \langle du(\cdot) ,du(e_i)\rangle du(e_i).$

About $ F $-symphonic map, one can get the Liouville type result via conservation law and monotonicity formula if we imposing energy finite condition, one can refer to \cite{han2014monotonicity,Han2021,han2013stability}. In \cite{cao2022liouville}, we obtained Liouville theorem of several generalized maps between Riemannian manifold,including $\phi$-$F $ symphonic map via conservation law method. In \cite{s5689}, we obtained stability of F harmonic map and $ F $-symphonic map with potential.

In this paper,  we aim to generalize \cite[Theorem 4.1]{cao2022liouville} to the case of $\phi$-$F $ symphonic map with free boundary. In Theorem \ref{thmabc}, we get Liouville type theorem of free boundary prolem .

% $u$ 的 $\Phi_{S, p, \varepsilon}$-能量密度 $e_{\Phi_{S, p, \varepsilon}}(u)$ 为
%$$
%e_{\Phi_{S, p, \epsilon}}(u)=\frac{1}{2 p}\left[\frac{m-2 p}{p^2}|d u|^{2 p}+m^{\frac{p}{2}-1}\left\|u^* h\right\|^p\right]+\frac{1}{4 \varepsilon^n}\left(1-|u|^2\right)^2
%$$

 In \cite{132132132}, the authors considered  $\Phi_{S, p, \varepsilon}$-energy
$$
\begin{aligned}
	E_{\Phi_{S, p, \varepsilon}}(u) 
	& =\int_M\left\{\frac{1}{2 p}\left[\frac{m-2 p}{p^2}|d u|^{2 p}+m^{\frac{p}{2}-1}\left\|u^* h\right\|^p\right]+\frac{1}{4 \varepsilon^n}\left(1-|u|^2\right)^2\right\} d v_g,
\end{aligned}
$$
 where $u^* h(X, Y)=h(d u(X), d u(Y))$. Its Euler-Lagrange equation is

$$ \tau_{\Phi_{S, p, \varepsilon}}(u)=\operatorname{div}\left(\sigma_{p, u}\right)+\frac{1}{\varepsilon^{n}}\left(1-|u|^{2}\right) u, $$
where $$
\sigma_{p, u}(X)=\sum_{j=1}^m m^{\frac{p}{2}-1}\left\|u^* h\right\|^{p-2} h\left(d u\left(e_j\right), d u(X)\right) d u\left(e_j\right)+\frac{m-2 p}{p^2}|d u|^{2 p-2} d u(X).
$$	
Inspired by this functional,  on metric measure space $ (M,g,e^{-\phi}dv_g) ,$  we consider the functional 
	$$
	\begin{aligned}
	E_{\Phi_{S, p, \epsilon},\phi}^R(u)
		& =\int_M\left\{\frac{1}{2 p}\left[\frac{m-2 p}{p^2}|d u|^{2 p}+m^{\frac{p}{2}-1}\left\|u^* h\right\|^p\right]+\frac{1}{4 \varepsilon^n}\left(1-|u|^2\right)^2\right\}e^{-\phi} \d v_g.
	\end{aligned}
	$$	
Its Euler-Lagrange equation is  
	 $$
	 \begin{aligned}
	 		\tau_{\Phi_{S, p, \varepsilon},\phi}(u)=0.
	 \end{aligned}
	 $$	 	
	 where $ \tau_{\Phi_{S, p, \varepsilon},\phi}(u) $ is given by 
	\begin{equation}\label{cccccc}
		\begin{aligned}
			\tau_{\Phi_{S, p, \varepsilon},\phi}(u)=\operatorname{div}\left(\sigma_{p, u}\right)+\frac{1}{\varepsilon^{n}}\left(1-|u|^{2}\right) u-\sigma_{p, u}(\nabla \phi).
		\end{aligned}
	\end{equation}
Such kind of map is coined as $ \phi $-$\Phi_{S, p, \varepsilon}$ harmonic map in this paper.	
In Theorem \ref{thmabc1}, we get Liouville type theorem of free boundary prolem for such kind of map .

In this paper, our main motivation is to extend the results in \cite{MR2766664} and  \cite{MR2050491} to $ \phi $-$ F $-symphonic map  with free boundary, $ \phi $-$ F $ harmonic map with free boundary,$ \phi $-$\Phi_{S, p, \varepsilon}$ harmonic map with free boundary in section 2, section 3, section 4 (see Theorem \ref{thmabc},Theorem \ref{thm2},Theorem \ref{thmabc1}), respectively.  	
	
	\section{$F$ symphonic map  }
	In this section, we will consider $F$-symphonic map with free boundary condition under asymptotic conditions .

	\begin{thm}\label{thm1a}
	 For $d_F < \frac{m}{4}$, let $u:\left(\mathbf{R}_{+}^{m}, f g_{0},e^{-\phi} d v_g\right) \rightarrow(N, h)$ be a $C^{1} $  $ \phi $-$F$--symphonic map with free boundary condition: $u\left(\partial \mathbf{R}_{+}^{m}\right) \subset S \subset N, \frac{\partial u}{\partial v}(x) \perp T_{u(x)} S$ for any $x \in \partial \mathbf{R}_{+}^{m}$, where $ g_0 $ is  the standard Euclidean metric on $ \mathbf{R}_{+}^{m} $ and $f$ is some positive function on $\mathbf{R}_{+}^{m}$ which satisfy
		\[
		(c+4d_F-m+\frac{\partial \phi}{\partial x_i}x_i) f(x) \leq (d_F-\frac{m}{2}) \left( \frac{\partial f}{\partial x_{i}} \cdot x_{i}\right) ,  
		\]
 for some constant  $ c>0.  $  If the energy $E_{F}(u)<\infty$, then $u$ must be a constant map. 
	\end{thm}

\begin{proof}
	For $t \geq 0$, we define a family $\left\{V_{t}\right\}: \mathbf{R}_{+}^{m} \rightarrow N$ of maps by $V_{t}(x):=u(t x)$ for $x \in \mathbf{R}_{+}^{m}$, and set
	\[
	\Phi(R, t):= \int_{B(R)} F(\frac{\left| V_{t}^{*} h\right|^{2}}{2}) e^{-\phi}\d v_g,
	\]
	where $B(R)=\mathbf{R}_{+}^{m} \cap\{x:|x| \leq R\}$.

	Let $ \xi=\left.\frac{\mathrm{d} V_{t}}{\mathrm{~d} t}\right|_{t=1}=du(r\frac{\partial}{\partial r} )$,  applying Green's theorem, we calculate
	$$
	\begin{aligned}
		&\left.\frac{\partial}{\partial t} \Phi(R, t)\right|_{t=1}=\int_{B(R)} \langle \nabla^u \xi , \sigma_{u,F} \rangle e^{-\phi}\d v_g\\ & =\int_{B(R)}\left\langle\mathrm{d}^{*}\left(\sigma_{u,F}\right)-F^\prime(\frac{|u^*h |^2}{4}) du(\nabla \phi), \mathrm{d} u\left(r \frac{\partial}{\partial r}\right)\right\rangle e^{-\phi}\d v_g \\
		& +R \int_{\partial B(R) \cap\left\{x_{m}>0\right\}}\left\langle \sigma_{u,F}\left(\frac{\partial}{\partial v}\right), \mathrm{d} u\left(\frac{\partial}{\partial r}\right)\right\rangle e^{-\phi}\sigma_{R} \\
		& +\int_{\partial B(R) \cap\left\{x_{m}>0\right\}} \left\langle\sigma_{u,F}(v),\xi \right\rangle e^{-\phi}\mathrm{d} x,
	\end{aligned}
	$$
where $\frac{\partial}{\partial v}=f^{-1} \frac{\partial}{\partial r}$. So we get

\begin{equation*}
	\begin{split}
		\left.\frac{\partial}{\partial t} \Phi(R, t)\right|_{t=1}\geq 0
	\end{split}
\end{equation*}
On the other hand,
\begin{equation*}
	\begin{split}
		\Phi(R, t):=& \int_{B(R)} F(\frac{\left| V_{t}^{*} h\right|^{2}}{2}) v_{g}\\
		&=\int_{B(R)} F(f^{-1}\left( (h_{\alpha \beta}(u(t x)) \frac{\partial u^{\alpha}(t x)}{\partial x_{i}} \frac{\partial u^{\beta}(t x)}{\partial x_{j}}\right) ^2) f^{\frac{m}{2}}(x) \mathrm{d}v_{g_0},\\
		&=t^{-m}\int_{B(tR)} F\bigg(f^{-1}(\frac{x}{t})t^4\left( (h_{\alpha \beta}(u( x)) \frac{\partial u^{\alpha}( x)}{\partial x_{i}} \frac{\partial u^{\beta}( x)}{\partial x_{j}}\right) ^2\bigg) f^{\frac{m}{2}}(\frac{x}{t}) e^{-\phi(\frac{x}{t})}\mathrm{d}v_{g_0}.
	\end{split}
\end{equation*}
Then ,we have 

$$
\begin{aligned}
	\frac{d}{dt} 	\Phi(R, t)&= -mt^{-m-1} \int_{  B(R) }F(f^{-1}(\frac{x}{t})t^4 A(x) )  f^{\frac{m}{2}}(x)e^{-\phi(\frac{x}{t})}\mathrm{d}v_{g_0}   \\
	&+t^{-m}\int_{  B(R) }F^{\prime}(f^{-1}(\frac{x}{t})t^4 A(x) ) \frac{d}{dt}(f^{-1}(\frac{x}{t})t^4)A(x)  f^{\frac{m}{2}}(x) e^{-\phi(\frac{x}{t})} \mathrm{d}v_{g_0}\\
	&-t^{-m}\int_{  B(R) }F(f^{-1}(\frac{x}{t})t^4 A(x) ) \frac{m}{2}f^{\frac{m}{2}-1}\frac{1}{t^2} \left( \frac{\partial f}{\partial x_i}x_i\right)e^{-\phi(\frac{x}{t})} \mathrm{d}v_{g_0}  \\
	&	-t^{-m}\int_{  B(R) }F(f^{-1}(\frac{x}{t})t^4 A(x) ) \frac{m}{2}f^{\frac{m}{2}-1} \left( -\frac{\partial \phi}{\partial x_i}x_i\right) e^{-\phi(\frac{x}{t})}\mathrm{d}v_{g_0}  \\
	&+t^{-m}\int_{ \partial B(R) } R^{m-1} F(f^{-1}(x)t^4 A(x) ) f^{\frac{m}{2}} e^{-\phi(\frac{x}{t})} \mathrm{d}v_{g_0}.
\end{aligned}
$$
Hence, 
\begin{equation*}
	\begin{split}
	\frac{d}{dt}\bigg|_{t=1}	\Phi(R, t)\leq -c 	\Phi(R, t)+R \frac{d}{d R}	\Phi(R, t).
	\end{split}
\end{equation*}
Thus, by the standard procedure as in \cite{MR2766664} or  \cite{MR2050491} ,  we get
\begin{equation*}
	\begin{split}
		\int_{B(R)} F(\frac{\left| u^{*} h\right|^{2}}{2}) v_{g}\geq  C (\frac{R}{R_0})^{c}.
	\end{split}
\end{equation*}
By contradiction argument, we can finish the proof easily.
\end{proof}
	
	\begin{thm}\label{thmabc}	Let $ g_0 $ is  the fixed Riemman metric on $ \mathbf{R}_{+}^{m} $.  Let   $ (N^n, h) $ be a Riemannian manifold. 
		Let  $ u: (\mathbf{R}_{+}^{m},g=f^2 g_0,e^{-\phi}\d v_g)\to (N,h)  $ be $C^2$  $ \phi $-$F$-symphonic map. $ F^\prime(\frac{\|u^*h\|^2}{4})<\infty, $  and also
		\begin{equation}\label{er2}
			\begin{split}
				\int_{R}^{\infty} \frac{1}{\left( \int_{\partial B(r)} e^{-\phi} f^{m-4}\d S_{g_0}\right) ^{\frac{1}{3}}}\d r\geq R^{- \frac{\sigma}{3}},
			\end{split}
		\end{equation}
	where  $ \d S_{g_0} $ is the volume element of the boundary and
	 $f$ is some positive function on $\mathbf{R}_{+}^{m}$ which satisfy
	\begin{equation}\label{er1}
		\begin{aligned}
			\left( c+4d_F-m+\frac{\partial \phi}{\partial x_i}x_i\right)  f(x) \leq (d_F-\frac{m}{2}) \left( 2f\frac{\partial f}{\partial x_{i}} \cdot x_{i}\right) , 
		\end{aligned}
	\end{equation}
 for some constant  $ c>0 $.Furtherly, $ u (x) \to p_0 $ as $ |x|\to \infty $, then $ u $  is a constant.
		
%		 Ginzburg-Landau type $\phi$-$F$ harmonic map coupled with $\phi$-$F$ symphonic map, satisfying $F^{\prime}\left(\frac{|d u|^{2}}{2}\right)$ $<+\infty,F^{\prime}\left(\frac{|u^{*}h|^{2}}{4}\right)$ $<+\infty$  and the $F$-lower degree $l_{F}>0$.
		%		
		%		 the critical point of 
		%		\begin{equation}\label{et}
			%			\begin{split}
				%				E(u)=\int_{M}\left( F(\frac{|du |^2}{2})+F(\frac{|u^{*}h|^2}{4})+\frac{1}{4\epsilon^n}(1-|u|^2)^2\right) e^{-\phi}d\nu_{g},
				%			\end{split}
			%			\end{equation} 	where $ F, G\in C^2(M), \phi >0. $ 
%		If
%		\begin{equation}\label{ball2}
%			\begin{split}
%				\int_{R}^{\infty} \frac{1}{\bigg[\left( \int_{\partial B(R)} e^{-\phi} f^{m-2}\d v_{g_0}\right) ^{\frac{1}{2}} \bigg]^{\frac{4}{3}}}\d r\geq R^{-\frac{\sigma}{3}}.
%			\end{split}
%		\end{equation}
		
		%If for sufficiently large R,  \begin{equation*}
			%			\begin{split}
				%				\bigg(\int_{R}^{\infty} \frac{1}{\text{Vol}^{\frac{1}{2p-1}}(\partial B_r)} dr\bigg)^{-1} \leq CR^{\frac{\sigma}{2p-1}},
				%			\end{split}
			%		\end{equation*}
%		Assume that there exists two
%		positive functions $h_1(r)$ and $h_2(r)$ such that
%		$$h_1(r)[g - dr \otimes dr] \leq \operatorname{Hess}(r)\leq  h_2(r)[g - dr \otimes dr].$$
%		Suppose that 
%		\begin{equation*}
%			\begin{split}
%				1+(m-1)rh_1(r)-4rd_Fph_2(r)\geq \sigma>1.  \quad 
%			\end{split}
%		\end{equation*}
%		We assume that  $ \vol_g(B(R)) = o(R^\sigma), \vol_g(\partial B(R)) \gg \frac{1}{4} $ 
%		$$ \partial_r\phi \geq 0 , f\geq 1,r\frac{\partial \log f }{\partial r}\geq 0, d_F \leq \frac{m}{4},  rh_2(r)\geq 1. $$

	\end{thm}
	\begin{proof}
		
		As in in \cite{MR2766664} and  \cite{MR2050491}, we firstly modify the $F$-harmonic map $u$ at boundary $\partial \mathbf{R}_{+}^m$.
		Since $\lim _{|x| \rightarrow \infty} u(x)=Q_0$, there exists a neighborhood $U_{r_0}=\left\{\left(x_1, \ldots, x_m\right)\right.$ : $\left.\left|x_m\right|<r_0\right\}$ of $\partial \mathbf{R}_{+}^m$ such that the image $U_{r_0} \cap \mathbf{R}_{+}^m$ of $u$ lies on the standard neighborhood $\mathcal{N}(S)$ of $S$,  Let $\bar{x}=\left(x_1, \ldots, x_{m-1}\right.$, $\left.-x_m\right)$ and $x=\left(x_1, \ldots, x_{m-1}, x_m\right)$, if $\bar{x} \in U_{r_0} \backslash \mathbf{R}_{+}^m$ is the reflection point of $x \in \mathbf{R}_{+}^m$, we project $u(x)$ onto $S$ along the minimal geodesic $\gamma$, denote by ${u}(x) \in S$, extending $\gamma$ to some point $Q$ such that $\operatorname{dist}(u(x), {u}(x))=\operatorname{dist}(Q, {u}(x))$, 

Then we define the reflection ${u}(x)$ as follows
		$$
		\begin{cases}\tilde{u}(x)=u(x), & x \in \mathbf{R}_{+}^m, \\ \tilde{u}(x)=Q=u(\bar{x}), & x \in U_{r_0} \backslash \mathbf{R}_{+}^m .\end{cases}
		$$
	
		  According to the arguments in part 4 of \cite{gulliver1987harmonic}, we know that ${u}: U_{r_0} \cup \mathbf{R}_{+}^m \rightarrow N$ is a smooth map.

		Here we modify the proof of Dong \cite[Proposition 4.1, Theorem 5.1]{MR3449358} and Han et al. \cite{132132132}. If $ u $ is not a  constant map, then  by the proof of  \autoref{thm1a} , there exists a constant $ c>0, $ such that 
		\begin{equation}\label{87ttt}
			\begin{split}
				E(u)\geq C(u)R^{c}, \text{as } \quad R \to \infty.
			\end{split}
		\end{equation} 
		
		Using the same notations as in  \cite{132132132} or  \cite[Proposition 4.1, Theorem 5.1]{MR3449358}. Choose a local coordinate neighbourhood $ (U, \varphi) $ of $ p_0 $ in $ N^n $, such that
		$  \varphi(p_0) = 0 $. The assumption that $u(x) \rightarrow p_0=0$ as $r(x) \rightarrow \infty$ implies that there exists $R_{1}$ such that for $r(x)>$ $R_{1}, u(x) \in U$, and
		\[
		\left(\frac{\partial h_{\alpha \beta(u)}}{\partial u^{\gamma}} y^{\gamma}+2 h_{\alpha \beta}(u)\right) \geq\left(h_{\alpha \beta}(u)\right) \text { for } r(x)>R_{1} .
		\]
		For $w \in C_{0}^{2}\left(R_+^{m}\cap U_{r_0} \backslash B\left(R_{1}\right), exp^{-1}_{p_0}(U)\right)$, we consider the variation $\tilde{u}+t w: M^{m} \rightarrow$ $N^{n}$, Since $\tilde{u}$ is  ${F}$-symphonic map, thus we have 
		\[
		\left.\frac{\mathrm{d}}{\mathrm{d} t}\right|_{t=0} E_{sym}(\tilde{u}+t \omega)=0.
		\]
		
%		defined as follows:
%		\[
%		(u+t w)(q)=\left\{\begin{array}{ll}
%			u(q) & \text { if } q \in B\left(R_{1}\right), \\
%			\varphi^{-1}[(\varphi(u)+t w)(q)] & \text { if } q \in M^{m} \backslash B\left(R_{1}\right)
%		\end{array}\right.
%		\]   

Computing directly, we have

		\begin{equation*}
			\begin{split}
							&{\int_{\mathbb{R}^m_{r_0} \backslash B\left(R_{1}\right)}  g_0^{ik} g_0^{jl}F^{\prime}\left(\frac{\left\|u^{*} h\right\|^{2}}{4}\right) }
				{\left(2 h_{a \beta}(u)       \frac{\partial \tilde{u}^{\alpha}}{\partial x_{i}}                                                        \frac{\partial \omega^{\beta}}{\partial x_{j}}+\frac{\partial h_{\alpha \beta}(\tilde{u})}{\partial y^{\zeta}} \omega^{\zeta}       \frac{\partial \tilde{u}^{\alpha}}{\partial x_{i}}                                                            \frac{\partial \tilde{u}^{\beta}}{\partial x_{j}}                                                       \right)}
				{h_{\gamma \zeta}       \frac{\partial \tilde{u}^{\gamma}}{\partial x_{k}}                                                        \frac{\partial u^{\zeta}}{\partial x_{l}} f^{m-4}(x) \mathrm{d} v_{g_{0}}}\\
			&=0.
			\end{split}
		\end{equation*}
	where $ \mathbb{R}^m_{r_0}=\{x| x\in \mathbb{R}^m,x_m>-r_0\}. $
	
Thus ,we have 
		\begin{equation*}
			\begin{split}
							&{\int_{\mathbb{R}^m_{r_0} \backslash B\left(R_{1}\right)}  g_0^{ik} g_0^{jl} F^{\prime}\left(\frac{\left\|u^{*} h\right\|^{2}}{4}\right) }
				{\left(2 h_{a \beta}(\tilde{u})       \frac{\partial \tilde{u}^{\alpha}}{\partial x_{i}}                                                        \frac{\partial \omega^{\beta}}{\partial x_{j}}+\frac{\partial h_{a \beta}({u})}{\partial y^{\theta}} \omega^{\theta}       \frac{\partial \tilde{u}^{\alpha}}{\partial x_{i}}                                                            \frac{\partial \tilde{u}^{\beta}}{\partial x_{j}}                                                       \right)}
				{h_{\gamma \xi}       \frac{\partial \tilde{u}^{\gamma}}{\partial x_{k}}                                                              \frac{\partial \tilde{u}^{\xi}}{\partial x_{l}}                                                        f^{m-4}(x)e^{-\phi} \mathrm{d} v_{g_{0}}}\\
			&=0.
			\end{split}
		\end{equation*}		
		For $0<\epsilon \leq 1$, we define
		\begin{equation}\label{equ1}
			\begin{aligned}
				\varphi_{\epsilon}(t)=\left\{\begin{array}{ll}
					1 & t \leq 1 \\
					1+\frac{1-t}{\epsilon} & 1<t<1+\epsilon \\
					0 & t \geq 1+\epsilon
				\end{array}\right.
			\end{aligned}
		\end{equation}
		and choose the Lipschitz function $\phi(r(x))$  as in \cite[Proposition 4.1, Theorem 5.1]{MR3449358} to be
		\begin{equation}\label{equ2}
			\begin{aligned}
					\phi(|x|)=\varphi_{\epsilon}\left(\frac{|x|}{R}\right)\left(1-\varphi_{r_0}\left(\frac{|x|}{R_{1}}\right)\right), R>2 R_{1}.
			\end{aligned}
		\end{equation}
	and	
	\begin{equation}\label{equ3}
		\begin{aligned}
			\Phi\left(x_{m}\right):=\left\{\begin{array}{ll}
				1, & 0 \leq x_{m} \\
				1+\frac{x_{m}}{s}, & -s<x_{m}<0 \\
				0, & -r_{0}<x_{m} \leq-s,
			\end{array}\right.
		\end{aligned}
	\end{equation}

		Take $  \omega= \phi(|x|) \Phi(x_m)\tilde{u}, $ then we have 
		
		\begin{equation*}
			\begin{split}
							&{\int_{\mathbb{R}^m_{r_0} \backslash B\left(R_{1}\right)} g_{0}^{i k} g_{0}^{j l} F^{\prime}\left(\frac{\left\|u^{*} h\right\|^{2}}{4}\right) }
				{\left(2 h_{a \beta}(\tilde{u})       \frac{\partial {\tilde{u}}^{\alpha}}{\partial x_{i}}                                                        \frac{\partial {\tilde{u}}^{\beta}}{\partial x_{j}}+\frac{\partial h_{a \beta}(\tilde{u})}{\partial y^{\theta}} {\tilde{u}}^{\theta}       \frac{\partial {\tilde{u}}^{\alpha}}{\partial x_{i}}                                                            \frac{\partial {\tilde{u}}^{\beta}}{\partial x_{j}}                                                       \right)}
				{h_{\gamma \xi}       \frac{\partial {\tilde{u}}^{\gamma}}{\partial x_{k}}                                                              \frac{\partial {\tilde{u}}^{\xi}}{\partial x_{l}}                                                        f^{m-4}(x) e^{-\phi}\mathrm{d} v_{g_{0}}}\\
											=&{\int_{\mathbb{R}^m_{r_0} \backslash B\left(R_{1}\right)} g_{0}^{i k} g_{0}^{j l} F^{\prime}\left(\frac{\left\|u^{*} h\right\|^{2}}{4}\right) }
				{\left(2 h_{a \beta}(\tilde{u})       \frac{\partial {\tilde{u}}^{\alpha}}{\partial x_{i}}                                                        {\tilde{u}}^{\beta}\frac{\partial (\phi(|x|) \Phi(x_m)) }{\partial x_{j}}\right)}
				{h_{\gamma \xi} \frac{\partial {\tilde{u}}^{\gamma}}{\partial x_{k}} \frac{\partial {\tilde{u}}^{\xi}}{\partial x_{l}} f^{m-4}(x) e^{-\phi}\mathrm{d} v_{g_{0}}}\bigg].
			\end{split}
		\end{equation*}	
	Let $ \nu $ be the outer normal vector field along $ \partial B(R) $,  we have 
		
	\begin{equation}\label{key1}
		\begin{aligned}
			&{ \int_{\mathbb{R}^m_{r_0} \cap (B(R) \backslash B\left(R_{2}\right))  } g_{0}^{i k} g_{0}^{j l} F^{\prime}\left(\frac{\left\|u^{*} h\right\|^{2}}{4}\right) }
			{\left(2 h_{a \beta}(\tilde{u})       \frac{\partial {\tilde{u}}^{\alpha}}{\partial x_{i}}                                                        \frac{\partial {\tilde{u}}^{\beta}}{\partial x_{j}}+\frac{\partial h_{a \beta}(\tilde{u})}{\partial y^{\theta}} {\tilde{u}}^{\theta}       \frac{\partial {\tilde{u}}^{\alpha}}{\partial x_{i}}                                                            \frac{\partial {\tilde{u}}^{\beta}}{\partial x_{j}}                                                       \right)}
			\\
			&\times {h_{\gamma \xi}       \frac{\partial {\tilde{u}}^{\gamma}}{\partial x_{k}}                                                              \frac{\partial {\tilde{u}}^{\xi}}{\partial x_{l}}                                                       \phi(|x|) \Phi(x_m)   f^{m-4}(x)e^{-\phi} \mathrm{d} v_{g_{0}}}+D(R_1)\\
			=&{\int_{ \mathbb{R}^m_{r_0} \cap \partial B(R)} g_{0}^{i k}  F^{\prime}\left(\frac{\left\|u^{*} h\right\|^{2}}{4}\right) }
			{\left(2 h_{a \beta}(\tilde{u})       \frac{\partial {\tilde{u}}^{\alpha}}{\partial x_{i}}                                                        \tilde{u}^{\beta}g_0^{jl}\frac{\partial r}{\partial x_j }\right)}
			{h_{\gamma \xi}       \frac{\partial {\tilde{u}}^{\gamma}}{\partial x_{k}}                                                              \frac{\partial \tilde{u}^{\xi}}{\partial x_{l}}                                                        \Phi(x_m) f^{m-4}(x)e^{-\phi} \mathrm{d} v_{g_{0}}}\bigg]\\
			&+{\int_{\mathbb{R}^m_{r_0} \backslash B\left(R_{1}\right)} g_{0}^{i k} g_{0}^{j l} F^{\prime}\left(\frac{\left\|u^{*} h\right\|^{2}}{4}\right) }
			{\left(2 h_{a \beta}(\tilde{u})       \frac{\partial {\tilde{u}}^{\alpha}}{\partial x_{i}}                                                        \tilde{u}^{\beta}\phi(|x|)\frac{\partial ( \Phi(x_m)) }{\partial x_{j}}\right)}
			{h_{\gamma \xi}       \frac{\partial {\tilde{u}}^{\gamma}}{\partial x_{k}}                                                              \frac{\partial {u}^{\xi}}{\partial x_{l}}                                                        f^{m-4}(x)e^{-\phi} \mathrm{d} v_{g_{0}}}\bigg]\\
			=&{\int_{ \mathbb{R}^m_{r_0} \cap \partial B(R)} g_{0}^{i k}  F^{\prime}\left(\frac{\left\|u^{*} h\right\|^{2}}{4}\right) }
			{\left(2 h_{a \beta}(\tilde{u})       \frac{\partial {\tilde{u}}^{\alpha}}{\partial x_{i}}                                                        {\tilde{u}}^{\beta}\nu^{l}\right)}
			{h_{\gamma \xi}       \frac{\partial {\tilde{u}}^{\gamma}}{\partial x_{k}}                                                              \frac{\partial {\tilde{u}}^{\xi}}{\partial x_{l}}                                                        \Phi(x_m) f^{m-4}(x) e^{-\phi}\mathrm{d} v_{g_{0}}}\bigg],\\
			&+{\int_{\mathbb{R}^m_{r_0} \backslash B\left(R_{1}\right)} g_{0}^{i k} g_{0}^{j l} F^{\prime}\left(\frac{\left\|u^{*} h\right\|^{2}}{4}\right) }
			{\left(2 h_{a \beta}(\tilde{u})       \frac{\partial {\tilde{u}}^{\alpha}}{\partial x_{i}}                                                        {\tilde{u}}^{\beta}\phi(|x|)\frac{\partial ( \Phi(x_m)) }{\partial x_{j}}\right)}
			{h_{\gamma \xi}       \frac{\partial {\tilde{u}}^{\gamma}}{\partial x_{k}}                                                              \frac{\partial {\tilde{u}}^{\xi}}{\partial x_{l}}                                                        f^{m-4}(x) e^{-\phi}\mathrm{d} v_{g_{0}}}\bigg]\\
		\end{aligned}
	\end{equation}	
where  we set 
	\begin{equation*}
		\begin{aligned}
				&D(R_1)\\
			=&{\int_{ \mathbb{R}^m_{r_0}  \cap\left( B(R_2) \backslash B\left(R_{1} \right) \right)} g_{0}^{i k} g_{0}^{j l} F^{\prime}\left(\frac{\left\|\tilde{u}                                                              ^{*} h\right\|^{2}}{4}\right) }
			{\left(2 h_{a \beta}({\tilde{u}                                                              })       \frac{\partial {\tilde{u}                                                              }^{\alpha}}{\partial x_{i}}                                                        \frac{\partial {\tilde{u}                                                              }^{\beta}}{\partial x_{j}}+\frac{\partial h_{a \beta}({\tilde{u}                                                              })}{\partial y^{\theta}} {\tilde{u}                                                              }^{\theta}       \frac{\partial {\tilde{u}                                                              }^{\alpha}}{\partial x_{i}}                                                            \frac{\partial {\tilde{u}                                                              }^{\beta}}{\partial x_{j}}                                                       \right)}
			\\
			&\times	{h_{\gamma \xi}       \frac{\partial {\tilde{u}                                                              }^{\gamma}}{\partial x_{k}}                                                              \frac{\partial {\tilde{u}                                                              }^{\xi}}{\partial x_{l}}                                                       \left(1-\varphi_{r_0}\left(\frac{|x|}{R_{1}}\right)\right) f^{m-4}(x) e^{-\phi}\mathrm{d} v_{g_{0}}}\\
			&-\bigg[ {\int_{\mathbb{R}^m_{r_0} \backslash B\left(R_{1}\right)} g_{0}^{i k} g_{0}^{j l} F^{\prime}\left(\frac{\left\|\tilde{u}                                                              ^{*} h\right\|^{2}}{4}\right) }
			{\left(2 h_{a \beta}(\tilde{u}                                                              )       \frac{\partial {\tilde{u}                                                              }^{\alpha}}{\partial x_{i}}                                                        {\tilde{u}                                                              }^{\beta}\frac{\partial \varphi_{r_0} (\frac{|x|}{R_1})}{\partial x_{j}}\right)}
			{h_{\gamma \xi}       \frac{\partial {\tilde{u}                                                              }^{\gamma}}{\partial x_{k}}                                                              \frac{\partial {\tilde{u}                                                              }^{\xi}}{\partial x_{l}}                                                         f^{m-4}(x)e^{-\phi} \mathrm{d} v_{g_{0}}}\bigg].
		\end{aligned}
	\end{equation*}
		%
		%let 	$$Z(R)=\int_{B(R)\backslash B(R_2)}fF(\frac{|du |^2}{2})+fF(\frac{|u^{*}h|^2}{4})+\frac{1}{4\epsilon^n}(1-|u|^2)^2d\nu_{g}+D(R_1),$$
Let $ s\to 0$,   $ \tilde{u} \to u ,$  we get 
		
		\begin{equation}\label{key3}
			\begin{aligned}
					&{ \int_{\mathbb{R}^m_{+} \cap (B(R) \backslash B\left(R_{2}\right))  } g_{0}^{i k} g_{0}^{j l} F^{\prime}\left(\frac{\left\|u^{*} h\right\|^{2}}{4}\right) }
				{\left(2 h_{a \beta}(u)       \frac{\partial {u}^{\alpha}}{\partial x_{i}}                                                        \frac{\partial {u}^{\beta}}{\partial x_{j}}+\frac{\partial h_{a \beta}(u)}{\partial y^{\theta}} {u}^{\theta}       \frac{\partial {u}^{\alpha}}{\partial x_{i}}                                                            \frac{\partial {u}^{\beta}}{\partial x_{j}}                                                       \right)}
				\\
				&\times	{h_{\gamma \xi}       \frac{\partial {u}^{\gamma}}{\partial x_{k}}                                                              \frac{\partial {u}^{\xi}}{\partial x_{l}}                                                       \phi(|x|) \Phi(x_m)   f^{m-4}(x)e^{-\phi} \mathrm{d} v_{g_{0}}}+D^*(R_1)\\
				=&{\int_{ \mathbb{R}^m_{+} \cap \partial B(R)} g_{0}^{i k}  F^{\prime}\left(\frac{\left\|u^{*} h\right\|^{2}}{4}\right) }
				{\left(2 h_{a \beta}(u)       \frac{\partial {u}^{\alpha}}{\partial x_{i}}                                                      {u}^{\beta}\nu^{l}\right)}
				{h_{\gamma \xi}       \frac{\partial {u}^{\gamma}}{\partial x_{k}}                                                              \frac{\partial {u}^{\xi}}{\partial x_{l}}                                                        \Phi(x_m) f^{m-4}(x)e^{-\phi} \mathrm{d} v_{g_{0}}}\bigg].\\
			\end{aligned}
		\end{equation}
	Here $ D^*\left(R_1\right) =\lim\limits_{s\to 0} D(R_1).$ 	Let 	
	\begin{equation*}
			\begin{split}
				Z(R)=&{\int_{\mathbb{R}^m_{r_0}  \cap\left( B(R) \backslash B\left(R_{2}\right) \right)} g_{0}^{i k} g_{0}^{j l} F^{\prime}\left(\frac{\left\|u^{*} h\right\|^{2}}{4}\right) }
				{\left(2 h_{a \beta}(u)       \frac{\partial {u}^{\alpha}}{\partial x_{i}}                                                        \frac{\partial {u}^{\beta}}{\partial x_{j}}\right)}
				{h_{\gamma \xi}       \frac{\partial {u}^{\gamma}}{\partial x_{k}}                                                              \frac{\partial {u}^{\xi}}{\partial x_{l}}                                                       e^{-\phi} f^{m-4}(x) \mathrm{d} v_{g_{0}}}+D^*(R_1).\\
						\end{split}
		\end{equation*}		
		The RHS of \eqref{key3} can be estimated as follows:  by Han \cite[(16)]{Han2021}, we have

		\begin{equation*}
			\begin{split}
				&{\int_{\mathbb{R}^m_{+} \cap \partial B(R)} g_{0}^{i k}  F^{\prime}\left(\frac{\left\|u^{*} h\right\|^{2}}{4}\right) }
				{\left(2 h_{a \beta}(u)       \frac{\partial {u}^{\alpha}}{\partial x_{i}}                                                        {u}^{\beta}\nu^{l}\right)}
				{h_{\gamma \xi}       \frac{\partial {u}^{\gamma}}{\partial x_{k}}                                                              \frac{\partial {u}^{\xi}}{\partial x_{l}}                                                        f^{m-4}(x) e^{-\phi}\mathrm{d} v_{g_{0}}}\bigg]\\
				\leq & \sqrt[4]{m}\left(\int_{\mathbb{R}^m_{+} \cap\partial B(R)} F^{\prime}\left(\frac{\left\|u^{*} h\right\|^{2}}{4}\right)\left\|u^{*} h\right\|^{2} f^{m-4}(x) e^{-\phi}\mathrm{d} S_{g_{0}}\right)^{\frac{3}{4}} \\
				&\times \left(\int_{\mathbb{R}^m_{+} \cap\partial B(R)} F^{\prime}\left(\frac{\left\|u^{*} h\right\|^{2}}{4}\right) \right. \left.\left(\sum_{\alpha, \beta=1}^{n} h_{\alpha \beta} u^{\alpha} u^{\beta}\right)^{2}  f^{m-4}(x)e^{-\phi} \mathrm{d} S_{g_{0}}\right)^{\frac{1}{4}}.
			\end{split}
		\end{equation*}		
%		By Dong\cite[(4.5)]{MR3449358}, we know that 
%		
%		
%		\begin{equation*}
%			\begin{split}
%				&\bigg[	\int_{\partial B(R)}   g_{0}^{kl}|du |^{-2}F^{\prime}\left(\frac{|d u|^{2}}{2}\right)\left[2 h_{\alpha \beta}(u)       \frac{\partial {u}^{\alpha}}{\partial x_{i}}                                                       {u}^{\beta}\nu^{j}\right]h_{\gamma \xi}       \frac{\partial {u}^{\gamma}}{\partial x_{k}}                                                              \frac{\partial {u}^{\xi}}{\partial x_{l}}                                                        e^{-\phi}    f^{m-4}(x) \d v_{g_{0}} \\
%				&\leq  \left(\int_{\partial B(R)} g_0^{ij} F^{\prime}(\frac{|du |^2}{2})h_{\alpha\beta}       \frac{\partial {u}^{\alpha}}{\partial x_{i}}                                                        \frac{\partial u^{\beta}}{\partial x_{i}} e^{-\phi} f^{m-2}dv_{g_0} \right)^{\frac{1}{2}} \left(\int_{\partial B(R)} F^{\prime}(\frac{|du |^2}{2})h_{\alpha\beta}u^{\alpha}u^{\beta} e^{-\phi} f^{m-2}dv_{g_0} \right)^{\frac{1}{2}} .
%			\end{split}
%		\end{equation*}
		
		Combining all these formulas together, by Han \cite[(16)]{Han2021}, we have 	

\[
Z^{\frac{4}{3}}(R) \leq 2^{\frac{4}{3}} m^{\frac{1}{3}} Z^{\prime}(R) M(R)
\]
where 
\[
M(R)=\left(\int_{\partial B(R)} F^{\prime}\left(\frac{\left\|u^{*} h\right\|^{2}}{4}\right) \right.
\left.\left(\sum_{a \beta=1}^{n} h_{\alpha \beta} u^{\alpha} u^{\beta}\right)^{2} f^{m-4}(x) e^{-\phi}\mathrm{d} S_{g_{0}}\right)^{\frac{1}{3}}.
\]
For $R_{4} \geqslant R \geqslant R_{3}$, we have
\[
\int_{R}^{R_{4}} \frac{Z^{\prime}(r)}{Z^{\frac{4}{3}}(r)} \mathrm{d} r \geqslant \frac{1}{2^{\frac{4}{3}} m^{\frac{1}{3}}} \int_{R}^{R_{4}} \frac{1}{M(r)} \mathrm{d} r .
\]
Let $R_{4} \rightarrow \infty$ , notice that  $Z(R)>0$, we have
\[
Z(R) \leqslant\left(3 \cdot 2^{\frac{4}{3}} m^{\frac{1}{3}}\right)^{3}\left(\frac{1}{\int_{R}^{\infty} \frac{1}{M(r)} \mathrm{d} r}\right)^{3}, R>R_{3} .
\]
As usually, we choose function $ \eta(R) $ such that 

(1) On $\left(R_3, \infty\right)$, $\eta(R)$ is decreasing  and when  $R \rightarrow \infty$, $m(R) \rightarrow 0$;

(2) $\eta(R) \geq \max\limits_{r(x)=R}\left(\sum_{a \beta=1}^{n} h_{\alpha \beta} u^{\alpha} u^{\beta}\right)^{2}$.

Since $F^{\prime}\left(\frac{\left\|u^{*} h\right\|^{2}}{4}\right)<\infty$  and when $r(x) \rightarrow \infty$ , $u(x) \rightarrow 0$, we have 
\[
M(R) \leqslant C_{3} \eta^{\frac{1}{3}}(R)\left(\int_{\partial B(R)} f^{m-4}(x)  e^{-\phi}\mathrm{d} S_{g_{0}}\right)^{\frac{1}{3}}.
\]

Then , we get
$$
\int_R^{\infty} \frac{1}{M(r)} d r \geq \frac{1}{\eta^{\frac{1}{
	3}}(R)} 	\int_{R}^{\infty} \frac{1}{\bigg[\left( \int_{\partial B(r)} e^{-\phi} f^{m-4}\d S_{g_0}\right) \bigg]^{\frac{1}{3}}}\d r \geq \frac{1}{\eta^{\frac{1}{3}}(R)} R^{-\frac{\sigma}{3}}.
$$

It implies that 
		\begin{equation*}
			\begin{split}
				Z(R) \leq C\eta(R) R^\sigma.
			\end{split}
		\end{equation*}
		However
		
		\begin{equation*}
			\begin{split}
				Z(R)
				\geq& \int_{B(R)\backslash B\left(R_{2}\right)}  F(\frac{|u^{*}h|^2}{4}) e^{-\phi} dv_g+D^*(R_1)\\
				\geq & C R^{\sigma}.
			\end{split}
		\end{equation*}
		Noticing that $\eta$ converges to zero, this gives a contradiction.

	\end{proof}

	\begin{thm}\label{369}
		 Let $u:\left(\mathbf{R}_{+}^{m}, f g_{0}, e^{-\phi} \mathrm{d} v_{g}\right) \rightarrow\left(N^{n}, h\right)$ be  $\phi$-$F$ symphonic map, and  the conditions \eqref{er1} and \eqref{er2} hold. For any $p \in N^{n}$, if there is an (nonempty) open neighbourhood $U_{p} \subset N^{n}$, such that the family of open sets $\left\{U_{p} | p \in N^{n}\right\}$ has the following property: for some $p \in N^{n}, u(x) \in U_{p}$ as $r(x) \rightarrow \infty$, then $u$ is a constant map.	
%		 satisfying $F^{\prime}\left(\frac{|d u|^{2}}{2}\right)<+\infty, F^{\prime}\left(\frac{\left|u^{*} h\right|^{2}}{4}\right)<+\infty$ and the $F$-lower degree $l_{F}>0$ 

	\end{thm}

\begin{proof}
	Along the same line as \cite[Theorem 5.4]{MR3449358}, we modify  the proof of  \autoref{thmabc}. As in \cite[Theorem 5.4]{MR3449358},  we can construct a family of coordinate neighbourhoods $\left\{U_{p} | p \in N^{n}\right\}$  such that
\[
\left(\frac{\partial h_{\alpha \beta(y)}}{\partial y^{\gamma}} y^{\gamma}+2 h_{\alpha \beta}(y)\right) \geq\left(h_{\alpha \beta}(y)\right)  \text { on } U_{p},
\]
and
\[
h_{\alpha \beta}(y) y^{\alpha} y^{\beta} \leq C_{p},
\]
for an arbitrary constant $C_{p}$ depending on $p$. 
\end{proof}
	
	\section{$F$ harmonic map}
In this section, we will consider generalized map without potential under asymptotic conditions .
\begin{thm}\label{58}
	For $d_F \leq \frac{m}{4}$, let $u:\left(\mathbf{R}_{+}^{m}, f g_{0}, e^{-\phi} \mathrm{d} v_{g}\right) \rightarrow(N, h)$ be a $C^{1}$ $ \phi $-$F$--harmonic map with free boundary condition: $u\left(\partial \mathbf{R}_{+}^{m}\right) \subset S \subset N, \frac{\partial u}{\partial v}(x) \perp T_{u(x)} S$ for any $x \in \partial \mathbf{R}_{+}^{m}$, where $f$ is some positive function on $\mathbf{R}_{+}^{m}$ which satisfy
	\[
	(c+4d_F-m) f(x) \leq (d_F-\frac{m}{2}) \left( \frac{\partial f}{\partial x_{i}} \cdot x_{i}-\frac{\partial \phi}{\partial x_i}x_i\right) , 
	\]
 for some constant  $c>0 $	If the energy $E_{F}(u)<\infty$, then $u$ must be a constant map.
\end{thm}
\begin{proof}
	For $t \geq 0$, we define a family $\left\{V_{t}\right\}: \mathbf{R}_{+}^{m} \rightarrow N$ of maps by $V_{t}(x):=u(t x)$ for $x \in \mathbf{R}_{+}^{m}$, and set
	\[
	\Phi(R, t):= \int_{B(R)} F(\frac{\left| dV_{t} \right|^{2}}{2}) e^{-\phi}\d v_g,
	\]
	where $B(R)=\mathbf{R}_{+}^{m} \cap\{x:|x| \leq R\}$. Then, applying Green's theorem, we calculate

	Let $ \xi=\left.\frac{\mathrm{d} V_{t}}{\mathrm{~d} t}\right|_{t=1}=du(r\frac{\partial}{\partial r} )$
	\begin{equation*}
		\begin{split}
			\left.\frac{\partial}{\partial t} \Phi(R, t)\right|_{t=1}=&\int_{B(R)} \langle \nabla^u \xi , \tau_{F} \rangle e^{-\phi}\d v_g \\ 
			=& \int_{B(R)}\left\langle\mathrm{d}^{*}\left(\tau_F(u)\right)-
			F^\prime du(\nabla\phi), \mathrm{d} u\left(r \frac{\partial}{\partial r}\right)\right\rangle e^{-\phi}\d v_g\\
			& +R \int_{\partial B(R) \cap\left\{x_{m}>0\right\}}\left\langle \sigma_{u,F}\left(\frac{\partial}{\partial v}\right), \mathrm{d} u\left(\frac{\partial}{\partial r}\right)\right\rangle e^{-\phi}\sigma_{R} \\
			& +\int_{\partial B(R) \cap\left\{x_{m}>0\right\}} \left\langle\sigma_{u,F}(v),\xi \right\rangle e^{-\phi} \mathrm{d} x,
		\end{split}
	\end{equation*}
	where $\frac{\partial}{\partial v}=f^{-1} \frac{\partial}{\partial r}.$
	
	So we get	
	\begin{equation*}
		\begin{split}
			\left.\frac{\partial}{\partial t} \Phi(R, t)\right|_{t=1}\geq 0.
		\end{split}
	\end{equation*}
	On the other hand,
	\begin{equation*}
		\begin{split}
			\Phi(R, t):=& \int_{B(R)} F(\frac{\left| V_{t}^{*} h\right|^{2}}{2})e^{-\phi} dv_{g}\\
		=	&\int_{B(R)} F(f^{-1}\left( (h_{\alpha \beta}(u(t x)) \frac{\partial u^{\alpha}(t x)}{\partial x_{i}} \frac{\partial u^{\beta}(t x)}{\partial x_{j}}\right) ) f^{\frac{m}{2}}(x) e^{-\phi} \mathrm{d}v_{g_0},\\
			=&t^{-m}\int_{B(tR)} F\bigg(f^{-1}(\frac{x}{t})t^2\left( (h_{\alpha \beta}(u( x)) \frac{\partial u^{\alpha}( x)}{\partial x_{i}} \frac{\partial u^{\beta}( x)}{\partial x_{j}}\right) \bigg) f^{\frac{m}{2}}(\frac{x}{t}) e^{-\phi(\frac{x}{t})}\mathrm{d}v_{g_0},
		\end{split}
	\end{equation*}
	
	Set 
	$$
	\begin{aligned}
		A(x)=(h_{\alpha \beta}(u( x)) \frac{\partial u^{\alpha}( x)}{\partial x_{i}} \frac{\partial u^{\beta}( x)}{\partial x_{j}}.
	\end{aligned}
	$$
	Then ,we have 
	
	\begin{equation}\label{}
		\begin{split}
			\frac{d}{dt} 	\Phi(R, t)=&-mt^{-m-1} \int_{  B(R) }F(f^{-1}(\frac{x}{t})t^2 A(x) )  f^{\frac{m}{2}}(x) e^{-\phi(\frac{x}{t})} \mathrm{d}v_{g_0}   \\
			&+t^{-m}\int_{  B(R) }F^{\prime}(f^{-1}(\frac{x}{t})t^2 A(x) ) \frac{d}{dt}(f^{-1}(\frac{x}{t})t^2)A(x)  f^{\frac{m}{2}}(x)e^{-\phi(\frac{x}{t})}\mathrm{d}v_{g_0}\\
			&-t^{-m}\int_{  B(R) }F(f^{-1}(\frac{x}{t})t^2 A(x) ) \frac{m}{2}f^{\frac{m}{2}-1}\frac{1}{t^2}\left(  \frac{\partial f}{\partial x_i}x_i-\frac{\partial \phi}{\partial x_i}x_i\right) e^{-\phi(\frac{x}{t})}\mathrm{d}v_{g_0}\\
			&-t^{-m}\int_{  B(R) }F(f^{-1}(\frac{x}{t})t^2 A(x) ) f^{\frac{m}{2}-1}(\frac{x}{t})\frac{1}{t^2}\left(  -\frac{\partial \phi}{\partial x_i}x_i\right) e^{-\phi(\frac{x}{t})}\mathrm{d}v_{g_0}\\
			&+t^{-m}\int_{ \partial B(R) } R^{m-1} F(f^{-1}(x)t^2 A(x) ) f^{\frac{m}{2}}  e^{-\phi(\frac{x}{t})}\mathrm{d}v_{g_0}.
		\end{split}
	\end{equation}	
	Thus,
	
	\begin{equation*}
		\begin{split}
			\frac{d}{dt}\bigg|_{t=1}	\Phi(R, t)\leq -c	\Phi(R, t)+R \frac{d}{d R}	\Phi(R, t)
	,	\end{split}
	\end{equation*}
	thus, we get	
	\begin{equation*}
		\begin{split}
			\int_{B(R)} F(\frac{\left| du \right|^{2}}{2}) v_{g}\geq  C (\frac{R}{R_0})^{c}.
		\end{split}
	\end{equation*}	
	By contradiction argument, we can finish the proof easily.
\end{proof}

\begin{thm}\label{thm2}	Let $ \left(\mathbf{R}_{+}^{m}, f^2 g_{0}, e^{-\phi} \mathrm{d} v_{g}\right) $ be a complete metric measure space with a pole $  x_0.$  Let   $ (N^n, h) $ be a Riemannian manifold. 
	Let  $ u: \left(\mathbf{R}_{+}^{m}, f^2 g_{0}, e^{-\phi} \mathrm{d} v_{g}\right) \to (N,h)  $ be   $ \phi $-$F$--harmonic map.
	%		
	%		 the critical point of 
	%		\begin{equation}\label{et}
		%			\begin{split}
			%				E(u)=\int_{M}\left( F(\frac{|du |^2}{2})+F(\frac{|u^{*}h|^2}{4})+\frac{1}{4\epsilon^n}(1-|u|^2)^2\right) e^{-\phi}d\nu_{g},
			%			\end{split}
		%			\end{equation} 	where $ F, G\in C^2(M), \phi >0. $ 
	If the following condition holds, 
	\begin{equation}\label{er3}
		\begin{split}
			\int_{R}^{\infty} \frac{1}{\left( \int_{\partial B(R)} e^{-\phi} f^{m-2}\d v_{g_0}\right)^{\frac{2}{3}}}\d r\geq R^{-\sigma}.
		\end{split}
	\end{equation}
and  $f$ is some positive function on $\mathbf{R}_{+}^{m}$ which satisfy
\begin{equation}\label{er4}
	\begin{aligned}
		(c+4d_F-m+\frac{\partial \phi}{\partial x_i}x_i) f(x) \leq (d_F-\frac{m}{2}) \left(2f \frac{\partial f}{\partial x_{i}} \cdot x_{i}\right) , 
	\end{aligned}
\end{equation}
for some constant  $\varepsilon>0 $.  Furtherly, $ u (x) \to p_0 $ as $ |x|\to \infty $, then $ u $  is a constant.	 
	%If for sufficiently large R,  \begin{equation*}
		%			\begin{split}
			%				\bigg(\int_{R}^{\infty} \frac{1}{\text{Vol}^{\frac{1}{2p-1}}(\partial B_r)} dr\bigg)^{-1} \leq CR^{\frac{\sigma}{2p-1}},
			%			\end{split}
		%		\end{equation*}
	%		Assume that there exists two
	%		positive functions $h_1(r)$ and $h_2(r)$ such that
	%		$$h_1(r)[g - dr \otimes dr] \leq \operatorname{Hess}(r)\leq  h_2(r)[g - dr \otimes dr].$$
	%		Suppose that 
	%		\begin{equation*}
		%			\begin{split}
			%				1+(m-1)rh_1(r)-4rd_Fph_2(r)\geq \sigma>1.  \quad 
			%			\end{split}
		%		\end{equation*}
	%		We assume that  $ vol_g(B(R)) = o(R^\sigma), vol_g(\partial B(R))  \frac{1}{4} $ 
	%		$$ \partial_r\phi \geq 0 , f\geq 1,r\frac{\partial \log f }{\partial r}\geq 0, d_F \leq \frac{m}{4},  rh_2(r)\geq 1. $$	
	
\end{thm}
\begin{proof}

	Modification of the $\Phi_{S, p, \epsilon}$ $u$ at boundary $\partial \mathbf{R}_{+}^m$.
%	Since $\lim _{|x| \rightarrow \infty} u(x)=Q_0$, there exists a neighborhood $U_{r_0}=\left\{\left(x_1, \ldots, x_m\right)\right.$ : $\left.\left|x_m\right|<r_0\right\}$ of $\partial \mathbf{R}_{+}^m$ such that the image $U_{r_0} \cap \mathbf{R}_{+}^m$ of $u$ lies on the standard neighborhood $\mathcal{N}(S)$ of $S$, that means, for every $y \in \mathcal{N}(S)$, there exists only one point $y^{\prime} \in S$ such that $y^{\prime}$ is a projection of $y$ along the unique geodesic minimizing the distance between two points $y$ and $y^{\prime}$. Let $\bar{x}=\left(x_1, \ldots, x_{m-1}\right.$, $\left.-x_m\right)$ and $x=\left(x_1, \ldots, x_{m-1}, x_m\right)$, if $\bar{x} \in U_{r_0} \backslash \mathbf{R}_{+}^m$ is the reflection point of $x \in \mathbf{R}_{+}^m$, we project $u(x)$ onto $S$ along the minimal geodesic $\gamma$, denote by ${u}(x) \in S$, extending $\gamma$ to some point $Q$ such that $\operatorname{dist}(u(x), {u}(x))=\operatorname{dist}(Q, {u}(x))$, then we define the reflection ${u}(x)$ as follows
%	$$
%	\begin{cases}{u}(x)=u(x), & x \in \mathbf{R}_{+}^m, \\ {u}(x)=Q=u(\bar{x}), & x \in U_{r_0} \backslash \mathbf{R}_{+}^m .\end{cases}
%	$$	
By the same argument as in the proof of \autoref{thmabc} , we define the reflection map  ${u}: U_{r_0} \cup \mathbf{R}_{+}^m \rightarrow N$ is a smooth map.
	
	Here we modify the proof of Dong \cite[Proposition 4.1, Theorem 5.1]{MR3449358} and Han et al. \cite{132132132}. If $ u $ is not a  constant map, then  by the proof of \autoref{58},  there exists a constant $ c>0, $
	\begin{equation}\label{87tt}
		\begin{split}
			E(u)\geq C(u)R^{c}, \text{as } \quad R \to \infty.
		\end{split}
	\end{equation} 
	
	Using the same notations as in  \cite{132132132} or  \cite[Proposition 4.1, Theorem 5.1]{MR3449358}. Choose a local coordinate neighbourhood $ (U, \varphi) $ of $ p_0 $ in $ N^n $, such that
	$  \varphi(p_0) = 0 $. The assumption that $u(x) \rightarrow 0$ as $r(x) \rightarrow \infty$ implies that there exists $R_{1}$ such that for $r(x)>$ $R_{1}, u(x) \in U$, and
	\[
	\left(\frac{\partial h_{\alpha \beta(u)}}{\partial u^{\gamma}} y^{\gamma}+2 h_{\alpha \beta}(u)\right) \geq\left(h_{\alpha \beta}(u)\right) \text { for } r(x)>R_{1} .
	\]
	For $w \in C_{0}^{2}\left(R_+^{m}\cap U_{r_0} \backslash B\left(R_{1}\right), exp^{-1}_{p_0}(U)\right)$, we consider the variation $\tilde{u}+t w: M^{m} \rightarrow$ $N^{n}$, Since $\tilde{u}$ is  ${F}$-symphonic map, thus we have 
\[
\left.\frac{\mathrm{d}}{\mathrm{d} t}\right|_{t=0} E_{F}(\tilde{u}+t \omega)=0.
\]
	Computing directly, we have 
	
	\begin{equation*}
		\begin{split}
			&	\int_{\mathbb{R}^m_{r_0} \backslash B\left(R_{1}\right)}  g_{0}^{i j} F^\prime\left(\frac{|d u|^{2}}{2}\right)\left[2 h_{\alpha \beta}(u)       \frac{\partial {\tilde{u}}^{\alpha}}{\partial x_{i}}                                                             \frac{\partial {\tilde{u}}^{\beta}}{\partial x_{j}}                                                       +\frac{\partial h_{\alpha \beta}({\tilde{u}})}{\partial y^{\gamma}} w^{\gamma}       \frac{\partial {\tilde{u}}^{\alpha}}{\partial x_{i}}                                                        \frac{\partial {\tilde{u}}^{\beta}}{\partial x_{j}}\right]  e^{-\phi} f^{m-2}(x) \d v_{g_{0}} \\
			&=0.
		\end{split}
	\end{equation*}
where $ \mathbb{R}^m_{r_0}=\{x| x\in \mathbb{R}^m,x_m>-r_0\}.$

	Thus ,we have 
	\begin{equation*}
		\begin{split}
				&\int_{\mathbb{R}^m_{r_0} \backslash B\left(R_{1}\right)}  g_{0}^{i j} g_{0}^{kl}|du |^{-2}F^\prime\left(\frac{|d u|^{2}}{2}\right)\left[2 h_{\alpha \beta}(u)       \frac{\partial {\tilde{u}}^{\alpha}}{\partial x_{i}}                                                             \frac{\partial {u}^{\beta}}{\partial x_{j}}                                                       +\frac{\partial h_{\alpha \beta}(u)}{\partial y^{\theta}} w^{\theta}       \frac{\partial {\tilde{u}}^{\alpha}}{\partial x_{i}}                                                            \frac{\partial {\tilde{u}}^{\beta}}{\partial x_{j}}                                                       \right]h_{\gamma \xi}       \frac{\partial {\tilde{u}}^{\gamma}}{\partial x_{k}}                                                              \frac{\partial {\tilde{u}}^{\xi}}{\partial x_{l}}                                                       f^{m-4}(x) e^{-\phi}\d v_{g_{0}} \\
			=&0.
		\end{split}
	\end{equation*}	
%	For $0<\epsilon \leq 1$, define
%	\[
%	\varphi_{\epsilon}(t)=\left\{\begin{array}{ll}
%		1 & t \leq 1 \\
%		1+\frac{1-t}{\epsilon} & 1<t<1+\epsilon \\
%		0 & t \geq 1+\epsilon
%	\end{array}\right.
%	\]
%	and choose the Lipschitz function $\phi(r(x))$ in \cite[Proposition 4.1, Theorem 5.1]{MR3449358}to be
%	\[
%	\psi(r(x))=\varphi_{\epsilon}\left(\frac{r(x)}{R}\right)\left(1-\varphi_{1}\left(\frac{r(x)}{R_{1}}\right)\right), R>2 R_{1}.
%	\]
	Take $  \omega= \psi(r)\Phi(x_m){u},$ where $ \psi(r) $ and $ \Phi(x_m) $ are given by \eqref{equ2}\eqref{equ3}.
	
	\begin{equation*}
		\begin{split}
			&	\int_{\mathbb{R}^m_{r_0} \backslash B\left(R_{1}\right)}  g_{0}^{i j} g_{0}^{kl}|du |^{-2}F^\prime\left(\frac{|d u|^{2}}{2}\right)\left[2 h_{\alpha \beta}({\tilde{u}})       \frac{\partial {\tilde{u}}^{\alpha}}{\partial x_{i}}                                                        \frac{\partial {\tilde{u}}^{\beta}}{\partial x_{j}}+\frac{\partial h_{\alpha \beta}(u)}{\partial y^{\theta}} \tilde{u}^{\theta}       \frac{\partial {\tilde{u}}^{\alpha}}{\partial x_{i}}                                                            \frac{\partial {\tilde{u}}^{\beta}}{\partial x_{j}}                                                       \right]h_{\gamma \xi}       \frac{\partial {\tilde{u}}^{\gamma}}{\partial x_{k}}                                                              \frac{\partial {\tilde{u}}^{\xi}}{\partial x_{l}}                                                        e^{-\phi}    f^{m-4}(x) \d v_{g_{0}} \\
			&=-\bigg[	\int_{\mathbb{R}^m_{r_0} \backslash B\left(R_{1}\right)}  g_{0}^{i j} g_{0}^{kl}|du |^{-2}F^\prime\left(\frac{|d u|^{2}}{2}\right)\left[2 h_{\alpha \beta}({u})       \frac{\partial {\tilde{u}}^{\alpha}}{\partial x_{i}}                                                       {\tilde{u}}^{\beta} \frac{\partial \psi  }{\partial x_{j}}\right]h_{\gamma \xi}       \frac{\partial {\tilde{u}}^{\gamma}}{\partial x_{k}}                                                              \frac{\partial {\tilde{u}}^{\xi}}{\partial x_{l}}                                                        e^{-\phi}    f^{m-4}(x) \d v_{g_{0}}. 
		\end{split}
	\end{equation*}	
	Using the definition of $ \psi  $, let $ \nu $ be the outer normal vector field along $ \partial B(R) $,  we have 
	
	\begin{equation}\label{key}
		\begin{aligned}
			&	\int_{\mathbb{R}^m_{r_0}\cap \left( B(R)\backslash B\left(R_{2}\right)\right) }  g_{0}^{i j} g_{0}^{kl}|du |^{-2}F^\prime\left(\frac{|d u|^{2}}{2}\right)\left[2 h_{\alpha \beta}(\tilde{u})       \frac{\partial \tilde{u}^{\alpha}}{\partial x_{i}}                                                        \frac{\partial \tilde{u}^{\beta}}{\partial x_{j}}+\frac{\partial h_{\alpha \beta}(u)}{\partial y^{\theta}} \tilde{u}^{\theta}       \frac{\partial \tilde{u}^{\alpha}}{\partial x_{i}}                                                            \frac{\partial \tilde{u}^{\beta}}{\partial x_{j}}                                                       \right]\\
			&  \times h_{\gamma \xi}       \frac{\partial \tilde{u}^{\gamma}}{\partial x_{k}}                                                              \frac{\partial \tilde{u}^{\xi}}{\partial x_{l}}                                                        e^{-\phi}   f^{m-4}(x) \d v_{g_{0}}
			+D(R_1)\\
			&=\bigg[	\int_{\mathbb{R}^m_{r_0}\cap \partial B(R)}   g_{0}^{kl}|du |^{-2}F^\prime\left(\frac{|d u|^{2}}{2}\right)\left[2 h_{\alpha \beta}(u)       \frac{\partial \tilde{u}^{\alpha}}{\partial x_{i}}                                                       \tilde{u}^{\beta}g_0^{ij}\frac{\partial r}{\partial x_j }\right]h_{\gamma \xi}       \frac{\partial \tilde{u}^{\gamma}}{\partial x_{k}}                                                              \frac{\partial \tilde{u}^{\xi}}{\partial x_{l}}                                                         e^{-\phi}    f^{m-4}(x) \d v_{g_{0}} \\
			&=\bigg[	\int_{ \mathbb{R}^m_{r_0}\cap\partial B(R)}   g_{0}^{kl}|du |^{-2}F^\prime\left(\frac{|d u|^{2}}{2}\right)\left[2 h_{\alpha \beta}(\tilde{u})       \frac{\partial\tilde{u}^{\alpha}}{\partial x_{i}}                                                       \tilde{u}^{\beta}\nu^{j}\right]h_{\gamma \xi}       \frac{\partial \tilde{u}^{\gamma}}{\partial x_{k}}                                                              \frac{\partial \tilde{u}^{\xi}}{\partial x_{l}}                                                         e^{-\phi}    f^{m-4}(x) \d v_{g_{0}} ,\\
		\end{aligned}
	\end{equation}	
	where 
	\begin{equation*}
		\begin{split}
			&D(R_1)\\
			=&		\int_{ \mathbb{R}^m_{r_0}\cap \left( B(R_2) \backslash B\left(R_{1}\right)\right) }  g_{0}^{i j} g_{0}^{kl}|du |^{-2}F^\prime\left(\frac{|d u|^{2}}{2}\right)\left[2 h_{\alpha \beta}(u)       \frac{\partial \tilde{u}^{\alpha}}{\partial x_{i}}                                                        \frac{\partial \tilde{u}^{\beta}}{\partial x_{j}}+\frac{\partial h_{\alpha \beta}(u)}{\partial y^{\theta}} \tilde{u}^{\theta}       \frac{\partial\tilde{u}^{\alpha}}{\partial x_{i}}                                                            \frac{\partial \tilde{u}^{\beta}}{\partial x_{j}}                                                       \right]\\
			&\quad\quad \times h_{\gamma \xi}       \frac{\partial \tilde{u}^{\gamma}}{\partial x_{k}}                                                              \frac{\partial \tilde{u}^{\xi}}{\partial x_{l}}                                                        e^{-\phi}    f^{m-4}(x) \d v_{g_{0}} \\
			%				&+{\int_{B(R_2) \backslash B\left(R_{1}\right)} g_{0}^{i k} g_{0}^{j l} F^{\prime}\left(\frac{\left\|u^{*} h\right\|^{2}}{4}\right) }
			%				{\left(2 h_{a \beta}(u)       \frac{\partial {u}^{\alpha}}{\partial x_{i}}                                                        \frac{\partial {u}^{\beta}}{\partial x_{j}}+\frac{\partial h_{a \beta}(u)}{\partial y^{\theta}} {u}^{\theta}       \frac{\partial {u}^{\alpha}}{\partial x_{i}}                                                            \frac{\partial {u}^{\beta}}{\partial x_{j}}                                                       \right)}
			%				{h_{\gamma \xi}       \frac{\partial {u}^{\gamma}}{\partial x_{k}}                                                              \frac{\partial {u}^{\xi}}{\partial x_{l}}                                                       e^{-\phi} f^{m-4}(x) \mathrm{d} v_{g_{0}}}\\
			%				&-\int_{B(R_2) \backslash B\left(R_{1}\right)} \frac{1}{\varepsilon^{n}}\left(1-|u|^{2}\right) \left( h_{\alpha\beta}u^{\alpha} {u}^{\beta}  + \frac{\partial h_{\alpha\beta}}{\partial y^\theta}{u}^\theta u^\alpha u^\beta \right) e^{-\phi}   \d v_g\\
			&-\bigg[	\int_{\mathbb{R}^m_{r_0}\cap\backslash B\left(R_{1}\right)}  g_{0}^{i j} g_{0}^{kl}|du |^{-2}F^\prime\left(\frac{|d u|^{2}}{2}\right)\left[2 h_{\alpha \beta}(\tilde{u})       \frac{\partial {u}^{\alpha}}{\partial x_{i}}                                                       \tilde{u}^{\beta} \frac{\partial \varphi_1 (\frac{r}{R_1})}{\partial x_{j}}\right]h_{\gamma \xi}       \frac{\partial \tilde{u}^{\gamma}}{\partial x_{k}}                                                              \frac{\partial \tilde{u}^{\xi}}{\partial x_{l}}                                                        e^{-\phi}    f^{m-4}(x) \d v_{g_{0}} .
		\end{split}
	\end{equation*}
	Let $s\to 0 $, noticing that $ \tilde{u} \to u$,   we get
	
			\begin{equation}\label{key1}
		\begin{aligned}
			&	\int_{\mathbb{R}^m_{+}\cap \left( B(R)\backslash B\left(R_{2}\right)\right) }  g_{0}^{i j} g_{0}^{kl}|du |^{-2}F^\prime\left(\frac{|d u|^{2}}{2}\right)\left[2 h_{\alpha \beta}(u)       \frac{\partial {u}^{\alpha}}{\partial x_{i}}                                                        \frac{\partial {u}^{\beta}}{\partial x_{j}}+\frac{\partial h_{\alpha \beta}(u)}{\partial y^{\theta}} {u}^{\theta}       \frac{\partial {u}^{\alpha}}{\partial x_{i}}                                                            \frac{\partial {u}^{\beta}}{\partial x_{j}}                                                       \right]\\
			&\times h_{\gamma \xi}       \frac{\partial {u}^{\gamma}}{\partial x_{k}}                                                              \frac{\partial {u}^{\xi}}{\partial x_{l}}                                                        e^{-\phi}   f^{m-4}(x) \d v_{g_{0}}
			+D^*(R_1)\\
%			&=\bigg[	\int_{\mathbb{R}^m_{r_0}\cap \partial B(R)}   g_{0}^{kl}|du |^{-2}F^\prime\left(\frac{|d u|^{2}}{2}\right)\left[2 h_{\alpha \beta}(u)       \frac{\partial {u}^{\alpha}}{\partial x_{i}}                                                       {u}^{\beta}g_0^{ij}\frac{\partial r}{\partial x_j }\right]h_{\gamma \xi}       \frac{\partial {u}^{\gamma}}{\partial x_{k}}                                                              \frac{\partial {u}^{\xi}}{\partial x_{l}}                                                         e^{-\phi}    f^{m-4}(x) \d v_{g_{0}} \\
			&=\bigg[	\int_{ \mathbb{R}^m_{+}\cap\partial B(R)}   g_{0}^{kl}|du |^{-2}F^\prime\left(\frac{|d u|^{2}}{2}\right)\left[2 h_{\alpha \beta}(u)       \frac{\partial {u}^{\alpha}}{\partial x_{i}}                                                       {u}^{\beta}\nu^{j}\right]h_{\gamma \xi}       \frac{\partial {u}^{\gamma}}{\partial x_{k}}                                                              \frac{\partial {u}^{\xi}}{\partial x_{l}}                                                         e^{-\phi}    f^{m-4}(x) \d v_{g_{0}} .
		\end{aligned}
	\end{equation}	
Here  $ D^*\left(R_1\right) =\lim\limits_{s\to 0} D(R_1) $.
		
	The RHS of \eqref{key1} can be estimated as follows: 
	%		 by Han \cite[(16)]{Han2021}
	%		
	%		
	%		\begin{equation*}
		%			\begin{split}
			%				&{\int_{\partial B(R)} g_{0}^{i k}  F^{\prime}\left(\frac{\left\|u^{*} h\right\|^{2}}{4}\right) }
			%				{\left(2 h_{a \beta}(u)       \frac{\partial {u}^{\alpha}}{\partial x_{i}}                                                        {u}^{\beta}\nu^{l}\right)}
			%				{h_{\gamma \xi}       \frac{\partial {u}^{\gamma}}{\partial x_{k}}                                                              \frac{\partial {u}^{\xi}}{\partial x_{l}}                                                       e^{-\phi} f^{m-4}(x) \mathrm{d} v_{g_{0}}}\bigg]\\
			%				&\leq  \sqrt[4]{m}\left(\int_{\partial B(R)} F^{\prime}\left(\frac{\left\|u^{*} h\right\|^{2}}{4}\right)\left\|u^{*} h\right\|^{2} e^{-\phi}f^{m-4}(x) \mathrm{d} S_{g_{0}}\right)^{\frac{3}{4}} \\
			%				&\times \left(\int_{\partial B(R)} F^{\prime}\left(\frac{\left\|u^{*} h\right\|^{2}}{4}\right) \right. \left.\left(\sum_{\alpha, \beta=1}^{n} h_{\alpha \beta} u^{\alpha} u^{\beta}\right)^{2} e^{-\phi} f^{m-4}(x) \mathrm{d} S_{g_{0}}\right)^{\frac{1}{4}}.
			%			\end{split}
		%		\end{equation*}		
	By Dong\cite[(4.5)]{MR3449358}, we know that

	\begin{equation*}
		\begin{split}
			&\bigg[	\int_{\mathbb{R}^m_{+}\cap\partial B(R)}   g_{0}^{kl}|du |^{-2}F^{\prime}\left(\frac{|d u|^{2}}{2}\right)\left[2 h_{\alpha \beta}(u)       \frac{\partial {u}^{\alpha}}{\partial x_{i}}                                                       {u}^{\beta}\nu^{j}\right]h_{\gamma \xi}       \frac{\partial {u}^{\gamma}}{\partial x_{k}}                                                              \frac{\partial {u}^{\xi}}{\partial x_{l}}                                                        e^{-\phi}    f^{m-4}(x) \d v_{g_{0}} \\
			\leq & \left(\int_{\mathbb{R}^m_{+}\cap\partial B(R)} g_0^{ij} F^{\prime}(\frac{|du |^2}{2})h_{\alpha\beta}       \frac{\partial {u}^{\alpha}}{\partial x_{i}}                                                        \frac{\partial u^{\beta}}{\partial x_{i}} e^{-\phi} f^{m-2}dv_{g_0} \right)^{\frac{1}{2}} \left(\int_{\partial B(R)} F^{\prime}(\frac{|du |^2}{2})h_{\alpha\beta}u^{\alpha}u^{\beta} e^{-\phi} f^{m-2}dv_{g_0} \right)^{\frac{1}{2}} .
		\end{split}
	\end{equation*}	
	Combining all these together, we have 		
	\begin{equation*}
		\begin{split}
			&Z(R)^2 \leq   Z^{\prime}(R)M(R) .
		\end{split}
	\end{equation*}	
Here \begin{equation*}
	\begin{aligned}
		M(R)=\left(\int_{\partial B(R)} F^{\prime}(\frac{|du |^2}{2})h_{\alpha\beta}u^{\alpha}u^{\beta}e^{-\phi} f^{m-2}dv_{g_0} \right) .
	\end{aligned}
\end{equation*}	
%We get
%
%\begin{equation*}
%	\begin{aligned}
%		Z(R)\leq \frac{4}{\int_R^{\infty} \frac{1}{M(r)} d r},\quad  \text{ for } R>R_3.
%	\end{aligned}
%\end{equation*}
	%		If $ Z^{\prime}(R)\leq 1, \sigma>1,$ then 	by \eqref{87t},	 for $ R >R_0,  $
	%		\begin{equation*}
		%			\begin{split}
			%				R-R_0+Z(R_0) \geq Z(R)\geq l_F E(u)-\int_{B(R) \backslash B\left(R_{2}\right)} \frac{1}{\varepsilon^{n}}\left(1-|u|^{2}\right)^2e^{-\phi} \d v_g +D(R_1)\\
			%				\geq l_F CR^{\sigma}-\int_{B(R) \backslash B\left(R_{2}\right)} \frac{1}{\varepsilon^{n}}\left(1-|u|^{2}\right)^2e^{-\phi} \d v_g +D(R_1),
			%			\end{split}
		%		\end{equation*}
	%		where $ E(u) $ is defined in (\ref{et}). This gives an contradiction. So, we may assume $ Z^{\prime}(R)> 1. $		
%	\begin{equation*}
%		\begin{split}
%			&Z(R) \leq  \sqrt[4]{m}\left(Z^{\prime}(R)\right)^{\frac{3}{4}}  M(R),
%		\end{split}
%	\end{equation*}
As usually, we choose function $ \eta(R) $ such that 

(1) On $\left(R_3, \infty\right)$, $\eta(R)$ is decreasing  and when  $R \rightarrow \infty$, $m(R) \rightarrow 0$;

(2) $\eta(R) \geq \max\limits_{r(x)=R}\left(\sum_{a \beta=1}^{n} h_{\alpha \beta} u^{\alpha} u^{\beta}\right)^{2}$.

\noindent It is not hard to see that \begin{equation}\label{pol1}
	\begin{split}
		%				&\leq  C \eta(R)^{\frac{1}{2}} \left[ \left( \int_{\partial B(R)}  e^{-\phi}f^{m-4}dv\right) ^{\frac{1}{4}}+\left( \int_{\partial B(R)} e^{-\phi} f^{m-2}\d v_{g_0}\right) ^{\frac{1}{2}}\right] \\
		%				&\leq  C \eta(R)^{\frac{1}{2}} \left[ 2\left( \int_{\partial B(R)} e^{-\phi} f^{m-2}\d v_{g_0}\right) ^{\frac{1}{2}}+\frac{1}{2}\right]\\
		M(R) \leq&  \eta(R)^{\frac{1}{2}}  2\left( \int_{\partial B(R)}  e^{-\phi}f^{m-2}\d v_{g_0}\right).
	\end{split}
\end{equation}
%Then, 
%$$
%\int_R^{\infty} \frac{1}{M(r)} d r \geq \frac{1}{\eta^{\frac{1}{2}}(R)} \int_R^{\infty} \frac{1}{Vol^{\frac{1}{2p}}\left((\partial B(r))\right.} d r \geq \frac{1}{\eta^{\frac{1}{2p-1}}(R)} R^{-\frac{C_0}{2p-1}},
%$$	
	%		Hence, we get 
	%		
	%		\begin{equation*}
		%			\begin{split}
			%				&Z(R)+\int_{B(R) \backslash B\left(R_{2}\right)} \frac{1}{4\varepsilon^{n}}\left(1-|u|^{2}\right)^2 e^{-\phi}   \d v_g\\
			%				&\leq  \sqrt[4]{m}\left(Z^{\prime}(R)+\int_{ \partial B(R) } \frac{1}{4\varepsilon^{n}}\left(1-|u|^{2}\right)^2 e^{-\phi}   \d v_g\right)^{\frac{3}{4}}  M(R).
			%			\end{split}
		%		\end{equation*}
%	So
%	\begin{equation*}
%		\begin{split}
%			&\frac{Z(R)^{2}}{Z(R)^{\prime}}  \leq 4  M(R).
%		\end{split}
%	\end{equation*}	
	Notice that 
	\begin{equation*}
		\begin{split}
			\int_{R}^{\infty} \frac{1}{\left( \int_{\partial B(R)}  e^{-\phi}f^{m-2} \d v_{g_0}\right) ^{\frac{2}{3}} }dr\geq R^{-\frac{\sigma}{3}}.
		\end{split}
	\end{equation*}	
	Then we can get 
	\begin{equation*}
		\begin{split}
			\int_{R}^{\infty} \frac{1}{M(R) }dr \geq\frac{C}{\eta(R)}  	\int_{R}^{\infty} \frac{1}{\bigg[\left( \int_{\partial B(R)}  e^{-\phi}f^{m-2}dv\right)  \bigg]^{\frac{2}{3}}}\d r \geq  \frac{C}{\eta(R)}R^{-\sigma},
		\end{split}
	\end{equation*}		
	and
	\begin{equation*}
		\begin{split}
		Z(R)^{\frac{-1}{3}}\geq \frac{1}{3\sqrt[4]{m}} 	\int_{R}^{\infty} \frac{1}{M^{\frac{4}{3}}(R) }dr
			\geq  \frac{C}{\eta^{\frac{2}{3}}(R)}R^{-\frac{\sigma}{3}}.
		\end{split}
	\end{equation*}
	It implies that 
	\begin{equation*}
		\begin{split}
			Z(R)\leq C\eta(R) R^\sigma.
		\end{split}
	\end{equation*}
	However,		
	\begin{equation*}
		\begin{split}
			Z(R)
			\geq &l_{F}\int_{B(R)\backslash B\left(R_{2}\right)}  F(\frac{|du |^2}{2}) dv_g+D^*(R_1)\\
			\geq & C R^{\sigma}.
		\end{split}
	\end{equation*}
	Noticing that $\eta$ converges to zero, this gives a contradiction.

\end{proof}
By the same argument as in Theorem \ref{369}, we have 
		\begin{thm}
		Let $u:\left(\mathbf{R}_{+}^{m}, f g_{0}, e^{-\phi} \mathrm{d} v_{g}\right) \rightarrow\left(N^{n}, h\right)$ be  $\phi$-$F$ harmonic map, and  the conditions \eqref{er3} and \eqref{er4} hold. For any $p \in N^{n}$, if there is an (nonempty) open neighbourhood $U_{p} \subset N^{n}$, such that the family of open sets $\left\{U_{p} | p \in N^{n}\right\}$ has the following property: for some $p \in N^{n}, u(x) \in U_{p}$ as $r(x) \rightarrow \infty$, then $u$ is a constant map.	
		%		 satisfying $F^{\prime}\left(\frac{|d u|^{2}}{2}\right)<+\infty, F^{\prime}\left(\frac{\left|u^{*} h\right|^{2}}{4}\right)<+\infty$ and the $F$-lower degree $l_{F}>0$ 

	\end{thm}

	\section{ $\Phi_{S, p, \varepsilon}$ harmonic map }
	In this section, we deal with $\Phi_{S, p, \varepsilon}$ harmonic map.
	
		\begin{thm}\label{thm3}
		For $p \in[2, m-2)$, let $u:\left(\mathbf{R}_{+}^{m}, f g_{0},e^{-\phi} \d v_g\right) \rightarrow(\mathbb{R}^n, h)$ be a $C^{1} $  $ \phi $-$\Phi_{S, p, \varepsilon}$ harmonic map with free boundary condition: $u\left(\partial \mathbf{R}_{+}^{m}\right) \subset S \subset N, \frac{\partial u}{\partial v}(x) \perp T_{u(x)} S$ for any $x \in \partial \mathbf{R}_{+}^{m}$, where $f$ is some positive function on $\mathbf{R}_{+}^{m}$ which satisfy
		\[
	\left(c+ 2p-m +\frac{\partial \phi}{\partial x_{i}} \cdot x_{i}\right) f(x) \leq (\frac{m}{2}-p) \frac{\partial f}{\partial x_{i}} \cdot x_{i} , 
		\]
	$\nabla_{\frac{\partial } {\partial r} } f \geq 0 $, $ c $ is a positive constant.	If the p-energy $E_{p,\phi}(u)<\infty$, then $u$ must be a constant map.
	\end{thm}
	
	\begin{proof}
		For $t \geq 0$, we define a family $\left\{V_{t}\right\}: \mathbf{R}_{+}^{m} \rightarrow N$ of maps by $V_{t}(x):=u(t x)$ for $x \in \mathbf{R}_{+}^{m}$, and set
%		\[
%		\Phi(R, t):= \int_{B(R)} F(\frac{\left| V_{t}^{*} h\right|^{2}}{2}) v_{g},
%		\]
\begin{equation*}
	\begin{split}
			E_{\Phi_{S, p, \varepsilon}}(u) &=\int_M\left\{\frac{1}{2 p}\left[\frac{m-2 p}{p^2}|d  V_{t}|^{2 p}+m^{\frac{p}{2}-1}\left\| V_{t}^* h\right\|^p\right]+\frac{1}{4 \varepsilon^n}\left(1-| V_{t}|^2\right)^2\right\} e^{-\phi} d v_g,
	\end{split}
\end{equation*}
		where $B(R)=\mathbf{R}_{+}^{m} \cap\{x:|x| \leq R\}$.	Let $ \xi=\left.\frac{\mathrm{d} V_{t}}{\mathrm{~d} t}\right|_{t=1}=du(r\frac{\partial}{\partial r} )$ , applying Green's theorem, we have

		\begin{equation*}
			\begin{split}
				\left.\frac{\partial}{\partial t} \Phi(R, t)\right|_{t=1}=&\int_{B(R)} \langle \nabla^u \xi , \sigma_{u,F} \rangle +\langle du(\xi),\frac{1}{\epsilon^n}(1-u^2)u\rangle e^{-\phi} \d v_g \\  =&\int_{B(R)}\left\langle \tau_{\Phi_{S, p, \varepsilon},\phi}(u), \mathrm{d} u\left(r \frac{\partial}{\partial r}\right)\right\rangle  e^{-\phi} \mathrm{d} x \\
				& +R \int_{\partial B(R) \cap\left\{x_{m}>0\right\}}\left\langle \sigma_{u,F}\left(\frac{\partial}{\partial v}\right), \mathrm{d} u\left(\frac{\partial}{\partial r}\right)\right\rangle e^{-\phi} \sigma_{R} \\
				& +\int_{\partial B(R) \cap\left\{x_{m}>0\right\}} \left\langle\sigma_{u,F}(v),\xi \right\rangle e^{-\phi} \mathrm{d} x,
			\end{split}
		\end{equation*}
		where $\frac{\partial}{\partial v}=f^{-1} \frac{\partial}{\partial r}$, $ \tau_{\Phi_{S, p, \varepsilon},\phi}(u) $ is given in  \eqref{cccccc}.		
		So we get
		$$
		\begin{aligned}
			\left.\frac{\partial}{\partial t} \Phi(R, t)\right|_{t=1}\geq 0.
		\end{aligned}
		$$

		On the other hand,
	$$
	\begin{aligned}
			& \int_{B(R)} \frac{1}{2p}m^{\frac{p}{2}-1}\left\| V_{t}^* h\right\|^p v_{g}\\
		&=\int_{B(R)} \frac{1}{2p}m^{\frac{p}{2}-1}\bigg[f^{-1}\left( (h_{\alpha \beta}(u(t x)) \frac{\partial u^{\alpha}(t x)}{\partial x_{i}} \frac{\partial u^{\beta}(t x)}{\partial x_{j}}\right) ^2)\bigg]^{\frac{p}{2}} f^{\frac{m}{2}}(x)e^{-\phi} \mathrm{d}v_{g_0},\\
		&=t^{-m}\int_{B(tR)} \frac{1}{2p}m^{\frac{p}{2}-1}\bigg(f^{-1}(\frac{x}{t})t^4\left( (h_{\alpha \beta}(u( x)) \frac{\partial u^{\alpha}( x)}{\partial x_{i}} \frac{\partial u^{\beta}( x)}{\partial x_{j}}\right) ^2\bigg)^{\frac{p}{2}} f^{\frac{m}{2}}(\frac{x}{t})e^{-\phi} \mathrm{d}v_{g_0},\\
		&=t^{2p-m}\int_{B(tR)}\frac{1}{2p}m^{\frac{p}{2}-1}\left( (h_{\alpha \beta}(u( x)) \frac{\partial u^{\alpha}( x)}{\partial x_{i}} \frac{\partial u^{\beta}( x)}{\partial x_{j}}\right) ^p f^{\frac{m-2p}{2}}(\frac{x}{t}) e^{-\phi(\frac{x}{t})} \mathrm{d}v_{g_0},
	\end{aligned}
	$$
		
	Thus, 	
		$$
		\begin{aligned}
			&\left.\frac{\partial}{\partial t}\right|_{t=1}\int_{B(R)} \frac{1}{2p} m^{\frac{p}{2}-1}\left\| V_{t}^* h\right\|^pe^{-\phi(x)}d v_{g_0}\\
			=& (2p-m)  \int_{B(R)}\frac{1}{2p}m^{\frac{p}{2}-1} f^{(m-2p) / 2}(x) {e}_{2p}(u) e^{-\phi(x)}\mathrm{d} v_{g_0} \\
			& -\frac{m-2p}{2} \int_{B(R)}\frac{1}{2p}m^{\frac{p}{2}-1}f^{(m-2p-2) / 2}(x) {e}_{2p}(u) \cdot\left(\frac{\partial f}{\partial x_{i}} \cdot x_{i}\right) e^{-\phi(x)}\mathrm{d} v_{g_0} \\
			+&\int_{B(tR)}\frac{1}{2p}m^{\frac{p}{2}-1}\left( (h_{\alpha \beta}(u( x)) \frac{\partial u^{\alpha}( x)}{\partial x_{i}} \frac{\partial u^{\beta}( x)}{\partial x_{j}}\right) ^p f^{\frac{m-2p}{2}}(x)  \frac{\partial \phi}{\partial x_{i}} \cdot x_{i} e^{-\phi(x)}\mathrm{d}v_{g_0}\\
			& +\int_{\partial B(R) \cap\left\{x_{m} \geq 0\right\}} R^{m-1}\frac{1}{2p}m^{\frac{p}{2}-1} f^{(m-2p) / 2}(x) {e}_{2p}(u). e^{-\phi(x)}\sigma_{R} \\
		\end{aligned}
				$$
 On the other hand, 				
				
				$$	
			\begin{aligned}
				&\int_M\left\{\frac{1}{2 p}\frac{m-2 p}{p^2}|dV_{t}|^{2p}\right\} e^{-\phi(x)}d v_g\\
				&=\frac{1}{2 p}t^{-m}\int_{B(tR)}\frac{m-2 p}{p^2} \bigg[f^{-1}(\frac{x}{t})t^2\left( (h_{\alpha \beta}(u( x)) \frac{\partial u^{\alpha}( x)}{\partial x_{i}} \frac{\partial u^{\beta}( x)}{\partial x_{i}}\right) \bigg]^{p} f^{\frac{m}{2}}(\frac{x}{t})e^{-\phi(x)} \mathrm{d}v_{g_0},\\
						&=\frac{1}{2 p}t^{2p-m}\int_{B(tR)} \frac{m-2 p}{p^2}\bigg[\left( (h_{\alpha \beta}(u( x)) \frac{\partial u^{\alpha}( x)}{\partial x_{i}} \frac{\partial u^{\beta}( x)}{\partial x_{i}}\right) \bigg]^{p} f^{\frac{m}{2}-p}(\frac{x}{t})e^{-\phi(x)} \mathrm{d}v_{g_0},
					\end{aligned}
				$$
		Thus,
		$$
		\begin{aligned}
		&\left.\frac{\partial}{\partial t}\right|_{t=1}	\int_M\left\{\frac{m-2 p}{p^2}\frac{1}{2 p}|dV_{t}|^{2p}\right\} e^{-\phi} d v_g \\
		&=  (2p-m)\int_{B(R)}\frac{1}{2p} \frac{m-2 p}{p^2} f^{(m-2p) / 2}(x) {e}_{2p}(u) e^{-\phi}\mathrm{d} x \\
			& -\frac{m-2p}{2} \int_{B(R)}\frac{1}{2p} \frac{m-2 p}{p^2}f^{(m-2p-2) / 2}(x) {e}_{2p}(u) \cdot\left(\frac{\partial f}{\partial x_{i}} \cdot x_{i}\right)e^{-\phi} \mathrm{d} x \\
			&+\int_{B(R)}\frac{1}{2p} \frac{m-2 p}{p^2} f^{(m-2p) / 2}(x) {e}_{2p}(u)\left( \frac{\partial \phi}{\partial x_{i}} \cdot x_{i}\right)  e^{-\phi}\mathrm{d} x \\
			& +R^{m-1} \int_{\partial B(R) \cap\left\{x_{m} \geq 0\right\}} \frac{1}{2p}\frac{m-2 p}{p^2} f^{(m-2p) / 2}(x) {e}_{2p}(u) e^{-\phi} \sigma_{R} .
					\end{aligned}
		$$
	While,
		
		$$
		\begin{aligned}
		\frac{1}{4 \varepsilon^n}	\int_M \left(1-| V_{t}|^2\right)^2e^{-\phi}\d v_g=\frac{1}{4 \varepsilon^n}t^{-m}\int_{  B(tR) } (1-|u|^2)^2 f^{\frac{m}{2}}(\frac{x}{t}) e^{-\phi(\frac{x}{t})}\d v_{g_0}.
					\end{aligned}
		$$		
		Thus, 		
		$$
		\begin{aligned}
			\left.\frac{\partial}{\partial t}\right|_{t=1}\int_M\frac{1}{4 \varepsilon^n}	 \left(1-| V_{t}|^2\right)^2e^{-\phi} dv_g &=-m\int_{  B(R) }\frac{1}{4 \varepsilon^n} (1-|u|^2)^2 f^{\frac{m}{2}}e^{-\phi}dv_{g_0}\\
			&-\int_{  B(R) }\frac{1}{4 \varepsilon^n} (1-|u|^2)^2 f^{\frac{m}{2}-1}\left( \frac{\partial f}{\partial x_{i}} \cdot x_{i}\right)  e^{-\phi} dv_{g_0}\\
			&+\int_{  B(R) }\frac{1}{4 \varepsilon^n} (1-|u|^2)^2 f^{\frac{m}{2}}\left( \frac{\partial \phi}{\partial x_{i}} \cdot x_{i}\right)  e^{-\phi} dv_{g_0}\\
			&+R^{m-1}\int_{  \partial B(R) } \frac{1}{4 \varepsilon^n}(1-|u|^2)^2 f^{\frac{m}{2}}dv_{g_0}.
		\end{aligned}
		$$
				Thus, 
		$$
		\begin{aligned}
				\left.\frac{\partial}{\partial t} \Phi(R, t)\right|_{t=1} \leq & -c \Phi(R, 1)+R \frac{\mathrm{d}}{\mathrm{d} R} \Phi(R, 1).
		\end{aligned}
		$$		
From this inequality , we deduce that 
		
		$$
		\begin{aligned}
			\int_{B(R)}\left\{\frac{1}{2 p}\left[\frac{m-2 p}{p^2}|d  u|^{2 p}+m^{\frac{p}{2}-1}\left\| u^* h\right\|^p\right]+\frac{1}{4 \varepsilon^n}\left(1-| V_{t}|^2\right)^2\right\}e^{-\phi} d v_g\geq C (\frac{R}{R_0})^{c}.
		\end{aligned}
		$$
		This is an contradiction to our assumption.
		\end{proof}
	
	\begin{thm}\label{thmabc1}	Let $ \left(\mathbf{R}_{+}^{m}, g=f^2 g_{0},e^{-\phi} \d v_g\right) $ be a complete metric measure space with a pole $  x_0,$  where $ g_0 $ is  the standard Euclidean metric on $ \mathbf{R}_{+}^{m}.$  Let   $ (N^n, h) $ be a Riemannian manifold. 
		Let  $ u: \left(\mathbf{R}_{+}^{m}, f g_{0}\right) \to (N,h)  $ be  $ \phi $-$ \Phi_{S, p, \epsilon} $ harmonic map ,
		Assume that 
		
	\begin{equation}\label{er5}
		\begin{aligned}
			\left( 	\int_R^{\infty} \frac{1}{ \left(  \int_{\partial B(R)} e^{-\phi} f^{m}\d v_{g_0}\right) ^{\frac{1}{2p-1}}} d r  \right)^{-1} \leq CR^{\frac{c}{2p-1}} ,
		\end{aligned}
	\end{equation}	
	and	
\begin{equation}\label{erf}
	\begin{aligned}
		\left(c+2p-m +\frac{\partial \phi}{\partial x_{i}} \cdot x_{i}\right) f^2(x) \leq (\frac{m}{2}-p) \frac{\partial f^2}{\partial x_{i}} \cdot x_{i} , 
	\end{aligned}
\end{equation}
for some constant $  \varepsilon>0 , p< \frac{m}{2}.$   Furtherly, $ u (x) \to p_0 $ as $ |x|\to \infty $, then $ u $  is a constant.
		%		
		%		 the critical point of 
		%		\begin{equation}\label{et}
			%			\begin{split}
				%				E(u)=\int_{M}\left( F(\frac{|du |^2}{2})+F(\frac{|u^{*}h|^2}{4})+\frac{1}{4\epsilon^n}(1-|u|^2)^2\right) e^{-\phi}d\nu_{g},
				%			\end{split}
			%			\end{equation} 	where $ F, G\in C^2(M), \phi >0. $ 
		%		If
		%		\begin{equation}\label{ball2}
			%			\begin{split}
				%				\int_{R}^{\infty} \frac{1}{\bigg[\left( \int_{\partial B(R)} e^{-\phi} f^{m-2}\d v_{g_0}\right) ^{\frac{1}{2}} \bigg]^{\frac{4}{3}}}\d r\geq R^{-\frac{\sigma}{3}}.
				%			\end{split}
			%		\end{equation}
		
		%If for sufficiently large R,  \begin{equation*}
			%			\begin{split}
				%				\bigg(\int_{R}^{\infty} \frac{1}{\text{Vol}^{\frac{1}{2p-1}}(\partial B_r)} dr\bigg)^{-1} \leq CR^{\frac{\sigma}{2p-1}},
				%			\end{split}
			%		\end{equation*}
		%		Assume that there exists two
		%		positive functions $h_1(r)$ and $h_2(r)$ such that
		%		$$h_1(r)[g - dr \otimes dr] \leq \operatorname{Hess}(r)\leq  h_2(r)[g - dr \otimes dr].$$
		%		Suppose that 
		%		\begin{equation*}
			%			\begin{split}
				%				1+(m-1)rh_1(r)-4rd_Fph_2(r)\geq \sigma>1.  \quad 
				%			\end{split}
			%		\end{equation*}
		%		We assume that  $ \vol_g(B(R)) = o(R^\sigma), \vol_g(\partial B(R)) \gg \frac{1}{4} $ 
		%		$$ \partial_r\phi \geq 0 , f\geq 1,r\frac{\partial \log f }{\partial r}\geq 0, d_F \leq \frac{m}{4},  rh_2(r)\geq 1. $$

	\end{thm}
\begin{proof}

%	Since $\lim _{|x| \rightarrow \infty} u(x)=Q_0$, there exists a neighborhood $U_{r_0}=\left\{\left(x_1, \ldots, x_m\right)\right.$ : $\left.\left|x_m\right|<r_0\right\}$ of $\partial \mathbf{R}_{+}^m$ such that the image $U_{r_0} \cap \mathbf{R}_{+}^m$ of $u$ lies on the standard neighborhood $\mathcal{N}(S)$ of $S$, that means, for every $y \in \mathcal{N}(S)$, there exists only one point $y^{\prime} \in S$ such that $y^{\prime}$ is a projection of $y$ along the unique geodesic minimizing the distance between two points $y$ and $y^{\prime}$. Let $\bar{x}=\left(x_1, \ldots, x_{m-1}\right.$, $\left.-x_m\right)$ and $x=\left(x_1, \ldots, x_{m-1}, x_m\right)$, if $\bar{x} \in U_{r_0} \backslash \mathbf{R}_{+}^m$ is the reflection point of $x \in \mathbf{R}_{+}^m$, we project $u(x)$ onto $S$ along the minimal geodesic $\gamma$, denote by ${u}(x) \in S$, extending $\gamma$ to some point $Q$ such that $\operatorname{dist}(u(x), {u}(x))=\operatorname{dist}(Q, {u}(x))$, 
%
%Modification of the $p$-harmonic map $u$ at boundary $\partial \mathbf{R}_{+}^m$. 

Suppose $ u $ is not constant, then  by the proof of Theorem \ref{thm3}, we get

\begin{equation}\label{158}
	\begin{aligned}
		E_{\Phi_{S, p, \epsilon},\phi}^R(u)=\int_{B(R)}\left\{\frac{1}{2 p}\left[\frac{m-2 p}{p^2}|d  V_{t}|^{2 p}+m^{\frac{p}{2}-1}\left\| V_{t}^* h\right\|^p\right]+\frac{1}{4 \varepsilon^n}\left(1-| V_{t}|^2\right)^2\right\} d v_g\geq C (\frac{R}{R_0})^{c}.
	\end{aligned}
\end{equation}

By the same argument as in the proof of \autoref{thmabc},  we define the reflection ${u}(x)$ as follows
	$$
	\begin{cases}\tilde{u}(x)=u(x), & x \in \mathbf{R}_{+}^m, \\ \tilde{u}(x)=Q=u(\bar{x}), & x \in U_{r_0} \backslash \mathbf{R}_{+}^m .\end{cases}
	$$
As before, we know that $\tilde{u}: U_{r_0} \cup \mathbf{R}_{+}^m \rightarrow N$ is a smooth map.  We follow the proof of \cite{132132132} and use the notations in \cite[Theorem 3.3]{132132132} .

%For $\omega \in C_0^2\left(M \backslash B\left(R_1\right), U\right)$, for sufficiently small $t$, we consider variation  $u+t \omega: M \rightarrow \mathbb{R}^n$ defined as follows
%$$
%(u+t \omega)(q)= \begin{cases}u(q), & q \in B\left(R_1\right) \\ (u+t \omega)(q), & q \in M \backslash B\left(R_1\right)\end{cases}
%$$

Choose a local coordinate neighbourhood $ (U, \varphi) $ of $ p_0 $ in $ N^n $, such that
$  \varphi(p_0) = 0 $.	For $w \in C_{0}^{2}\left(R_+^{m}\cap U_{r_0} \backslash B\left(R_{1}\right), exp^{-1}_{p_0}(U)\right)$, we consider the variation $\tilde{u}+t w: M^{m} \rightarrow$ $N^{n}$,
Since $\tilde{u}$ is $ \phi$-$\Phi_{S, p, \varepsilon}$-harmonic map, thus
$$
\left.\frac{d}{d t} E_{\Phi_{S, p, \varepsilon},\phi}(\tilde{u}+t \omega)\right|_{t=0}=0,
$$
which is 

$$
\begin{aligned}
0=&\int_{\mathbb{R}^m_{r_0} \backslash B\left(R_1\right)} m^{\frac{p}{2}-1}\left\|u^* h\right\|^{p-2} g^{i k} g^{j l} \frac{\partial \tilde{u}_\alpha}{\partial x_i} \frac{\partial \omega_\alpha}{\partial x_j} \frac{\partial \tilde{u}_\beta}{\partial x_k} \frac{\partial \tilde{u}_\beta}{\partial x_l}\\
	& +\int_{\mathbb{R}^m_{r_0} \backslash B\left(R_1\right)} \frac{m-2 p}{p^2}|d u|^{2 p-2} g^{i j} \frac{\partial \tilde{u}_\alpha}{\partial x_i} \frac{\partial \omega_\alpha}{\partial x_j} e^{-\phi} f^{m-2} d v_{g_0} \\
	& -\int_{\mathbb{R}^m_{r_0}\backslash B\left(R_1\right)} \frac{1}{\varepsilon^n}\left(1-|u|^2\right) u_\alpha \omega_\alpha e^{-\phi}d v_g ,
\end{aligned}
$$
where $ \mathbb{R}^m_{r_0}=\{x| x\in \mathbb{R}^m,x_m>-r_0\}. $

 $\phi(t) \in C_0^{\infty}\left(R_1, \infty\right)$. We  take $  \omega= \phi(r(x)) \Phi(x_m)\hat{u}, $,  $\hat{u}_\alpha=\frac{\tilde{u}_\alpha^2-c_\alpha^2}{\tilde{u}_\alpha}$, where $ \psi(r) $ and $ \Phi(x_m) $ are given by \eqref{equ2}\eqref{equ3}. Notice that  the noatations $ \tilde{u} $ in \cite[Theorem 3.3]{132132132} is $ \hat{u} $ here.  we have

\begin{equation}\label{ddd}
	\begin{aligned}
			& \int_{\mathbb{R}_{r_0}^m \backslash B\left(R_1\right)} m^{\frac{p}{2}-1}\left\|u^* h\right\|^{p-2} g^{i k} g^{j l} \frac{\partial \tilde{u}_\alpha}{\partial x_i} \frac{\partial \hat {u}_\alpha}{\partial x_j} \frac{\partial \tilde{u}_\beta}{\partial x_k} \frac{\partial \tilde{u}_\beta}{\partial x_l} \phi(r(x)) e^{-\phi} f^{m-4} d v_{g_0} \\
		& +\int_{\mathbb{R}_{r_0}^m\backslash B\left(R_1\right)} \frac{m-2 p}{p^2}|d u|^{2 p-2} g^{i j} \frac{\partial \tilde{u}_\alpha}{\partial x_i} \frac{\partial \hat {u}_\alpha}{\partial x_j}  \phi(r(x)) \Phi(x_m) e^{-\phi} f^{m-2} d v_{g_0} \\
		& -\int_{\mathbb{R}_{r_0}^m\backslash B\left(R_1\right)} \frac{1}{\varepsilon^n}\left(1-|u|^2\right) \tilde{u}_\alpha \hat{u}_\alpha  \phi(r(x)) \Phi(x_m) e^{-\phi} d v_g \\
		=& -\int_{\mathbb{R}_{r_0}^m \backslash B\left(R_1\right)} m^{\frac{p}{2}-1}\left\|u^* h\right\|^{p-2} g^{i k} g^{j l} \frac{\partial \tilde{u}_\alpha}{\partial x_i} \frac{\partial ( \phi(r(x)) \Phi(x_m)))}{\partial x_j} \frac{\partial \tilde{u}_\beta}{\partial x_k} \frac{\partial \tilde{u}_\beta}{\partial x_l} \hat{u}_\alpha e^{-\phi} f^{m-4} d v_{g_0}\\
		& -\int_{\mathbb{R}_{r_0}^m\backslash B\left(R_1\right)} \frac{m-2 p}{p^2}|d u|^{2 p-2} g^{i j} \frac{\partial \tilde{u}_\alpha}{\partial x_i} \frac{\partial ( \phi(r(x)) \Phi(x_m)))}{\partial x_j} \hat{u}_\alpha e^{-\phi} f^{m-2} d v_{g_0} .
	\end{aligned}
\end{equation}
%
%For $0 \leq \theta \leq 1$, we define
%$$
%\varphi_\theta(t)= \begin{cases}1, & t \leq 1 \\ 1+\frac{1-t}{\theta}, & 1<t<1+\theta \\ 0, & t \geq 1+\theta\end{cases}
%$$
%Choose  Lipschitz function $\phi(r(x))$ as follows:
%$$
%\phi(r(x))=\varphi_\theta\left(\frac{r(x)}{R}\right)\left(1-\varphi_1\left(\frac{r(x)}{R_1}\right)\right), R>2 R_1 .
%$$
 The left hand side of (\ref{ddd})  can be rewritten as

\begin{equation}\label{}
	\begin{aligned}
		& \int_{\mathbb{R}_{r_0}^m \cap    \left( B\left(R_2\right) \backslash B\left(R_1\right)\right)    } m^{\frac{p}{2}-1}\left\|u^* h\right\|^{p-2} g^{i k} g^{j l} \frac{\partial \tilde{u}_\alpha}{\partial x_i} \frac{\partial \hat {u}_\alpha}{\partial x_j} \frac{\partial \tilde{u}_\beta}{\partial x_k} \frac{\partial \tilde{u}_\beta}{\partial x_l}\left(1-\varphi_1\left(\frac{r(x)}{R_1}\right)\right) e^{-\phi} f^{m-4} d v_{g_0} \\
		& +\int_{\mathbb{R}_{r_0}^m \cap    (B\left(R_2\right) \backslash B\left(R_1\right))} \frac{m-2 p}{p^2}|d u|^{2 p-2} g^{i j} \frac{\partial \tilde{u}_\alpha}{\partial x_i} \frac{\partial \hat {u}_\alpha}{\partial x_j}\left(1-\varphi_1\left(\frac{r(x)}{R_1}\right)\right)e^{-\phi} f^{m-2} d v_{g_0}\\
		& -\int_{\mathbb{R}_{r_0}^m \cap  (B\left(R_2\right) \backslash B\left(R_1\right))  } \frac{1}{\varepsilon^n}\left(1-|u|^2\right) u_\alpha {u}_\alpha\left(1-\varphi_1\left(\frac{r(x)}{R_1}\right)\right) e^{-\phi}  d v_{g} \\
		& +\int_{  \mathbb{R}_{r_0}^m \cap  \left( B(R) \backslash B\left(R_2\right)\right)  } m^{\frac{p}{2}-1}\left\|u^* h\right\|^{p-2} g^{i k} g^{j l} \frac{\partial \tilde{u}_\alpha}{\partial x_i} \frac{\partial \hat {u}_\alpha}{\partial x_j} \frac{\partial u_\beta}{\partial x_k} \frac{\partial \tilde{u}_\beta}{\partial x_l} e^{-\phi} f^{m-4} d v_{g_0} \\
		& +\int_{  \mathbb{R}_{r_0}^m \cap   \left( B(R) \backslash B\left(R_2\right)\right)  } \frac{m-2 p}{p^2}|d u|^{2 p-2} g^{i j} \frac{\partial \tilde{u}_\alpha}{\partial x_i} \frac{\partial \hat {u}_\alpha}{\partial x_j} e^{-\phi} f^{m-2} d v_{g_0}\\
		& -\int_{ \mathbb{R}_{r_0}^m \cap    \left( B(R) \backslash B\left(R_2\right)\right) } \frac{1}{\varepsilon^n}\left(1-|u|^2\right) \tilde{u}_\alpha \hat{u}_\alpha e^{-\phi} d v_g \\
		& +\int_{ \mathbb{R}_{r_0}^m \cap   \left( B((1+\theta) R) \backslash B(R) \right) } m^{\frac{p}{2}-1}\left\|u^* h\right\|^{p-2} g^{i k} g^{j l} \frac{\partial \tilde{u}_\alpha}{\partial x_i} \frac{\partial \hat {u}_\alpha}{\partial x_j} \frac{\partial \tilde{u}_\beta}{\partial x_k} \frac{\partial \tilde{u}_\beta}{\partial x_l} \varphi_\theta\left(\frac{r(x)}{R}\right) e^{-\phi} f^{m-4} d v_{g_0} \\
		& +\int_{ \mathbb{R}_{r_0}^m \cap     \left( B((1+\theta) R) \backslash B(R) \right) } \frac{m-2 p}{p^2}|d u|^{2 p-2} g^{i j} \frac{\partial \tilde{u}_\alpha}{\partial x_i} \frac{\partial \hat {u}_\alpha}{\partial x_j} \varphi_\theta\left(\frac{r(x)}{R}\right) e^{-\phi} f^{m-2} d v_{g_0} \\
		&-\int_{ \mathbb{R}_{r_0}^m \cap   \left( B\langle\{1+\theta\} R\rangle \backslash B(R)\} \right) } \frac{1}{\epsilon^n}\left(1-|u|^2\right) \tilde{u}_\alpha \hat{u}_\alpha \psi_0\left(\frac{r(x)}{R}\right)e^{-\phi} dv_g,
	\end{aligned}
\end{equation}
Here $R_2=2 R_1$. Let $\theta \rightarrow 0$,  the right hand side of  $(\ref{ddd})$  becomes

\begin{equation}\label{ccc}
	\begin{aligned}
		& \int_{ \mathbb{R}_{r_0}^m \cap (B(R)\backslash B\left(R_2\right))} m^{\frac{p}{2}-1}\left\|u^* h\right\|^{p-2} g_0^{i k} g_0^{j i} \frac{\partial \tilde{u}_\alpha}{\partial x_i} \frac{\partial \hat {u}_\alpha}{\partial x_j} \frac{\partial \tilde{u}_\beta}{\partial x_k} \frac{\partial \tilde{u}_\beta}{\partial x_l} e^{-\phi} f^{m-4} d v_{g_0} \\
		& +\int_{ \mathbb{R}_{r_0}^m \cap (B(R)\backslash B\left(R_2\right))} \frac{m-2 p}{p^2}|d u|^{2 p-2} g_0^{i j} \frac{\partial \tilde{u}_\alpha}{\partial x_i} \frac{\partial \tilde{u}_\alpha}{\partial x_j} e^{-\phi} f^{m-2} d v_{g_0} \\
		& -\int_{ \mathbb{R}_{r_0}^m \cap (B(R)\backslash B\left(R_2\right))} \frac{1}{\epsilon^n}\left(1-|u|^2\right) \tilde{u}_\alpha \hat{u}_{\alpha}e^{-\phi}dv_g+D\left(R_1\right) \\
		 =&\int_{\partial B(R)\cap \mathbb{R}_{r_0}^m} m^{\frac{1}{2}-1}\left\|u^* h\right\|^{p-2} g^{j k} g^i \frac{\partial \tilde{u}_\alpha}{\partial x_i} \frac{\partial r(x)}{\partial x_j} \frac{\partial \tilde{u}_\beta}{\partial x_k} \frac{\partial \tilde{u}_\beta}{\partial x_i} \hat{u}_{\alpha} e^{-\phi} f^{m-4} d s \\
		& +\int_{\partial B(R)\cap \mathbb{R}_{r_0}^m} \frac{m-2 p}{p^2}|d u|^{2 p-2} g^{i j} \frac{\partial \tilde{u}_\alpha}{\partial x_i} \frac{\partial r(x)}{\partial x_j} \hat{u}_{\alpha} e^{-\phi} f^{m-2} d s, \\
		&-\int_{\mathbb{R}_{r_0}^m \backslash B\left(R_1\right)} m^{\frac{p}{2}-1}\left\|u^* h\right\|^{p-2} g^{i k} g^{j l} \frac{\partial \tilde{u}_\alpha}{\partial x_i} \phi(r(x)\frac{\partial \Phi(x_m)))}{\partial x_j} \frac{\partial \tilde{u}_\beta}{\partial x_k} \frac{\partial \tilde{u}_\beta}{\partial x_l} \hat{u}_\alpha e^{-\phi} f^{m-4} d v_{g_0} \\
		& -\int_{\mathbb{R}_{r_0}^m\backslash B\left(R_1\right)} \frac{m-2 p}{p^2}|d u|^{2 p-2} g^{i j} \frac{\partial \tilde{u}_\alpha}{\partial x_i} \phi(r(x))\frac{\partial (  \Phi(x_m)))}{\partial x_j} \hat{u}_\alpha e^{-\phi} f^{m-2} d v_{g_0},
	\end{aligned}
\end{equation}
where $D\left(R_1\right)$ is
$$
\begin{aligned}
	& D\left(R_1\right)=\int_{\mathbb{R}_{r_0}^m \cap (B(R)\backslash B\left(R_2\right))} m^{\frac{p}{2}-1}\left\|u^* h\right\|^{p-2} g^{i k} g^{i l} \frac{\partial \tilde{u}_\alpha}{\partial x_i} \frac{\partial \hat {u}_\alpha}{\partial x_j} \frac{\partial \tilde{u}_j}{\partial x_k} \frac{\partial \tilde{u}_\beta}{\partial x_l}\left(1-\varphi_1\left(\frac{r(x)}{R_1}\right)\right) e^{-\phi} f^{m-4} d v_{g_0} \\
	& +\int_{\mathbb{R}_{r_0}^m \cap (B(R)\backslash B\left(R_2\right))} \frac{m-2 p}{p^2}|d u|^{2 p-2} g^{i j} \frac{\partial \tilde{u}_\alpha}{\partial x_i} \frac{\partial \tilde{u}_\alpha}{\partial x_j}\left(1-\varphi_1\left(\frac{r(x)}{R_1}\right)\right) e^{-\phi} f^{m-2} d v_{g_0} \\
	& -\int_{\mathbb{R}_{r_0}^m \cap (B(R)\backslash B\left(R_2\right))} \frac{1}{\varepsilon^n}\left(1-|v|^2\right) \tilde{u}_\alpha \hat{u}_\alpha\left(1-\varphi_1\left(\frac{r(x)}{R_1}\right)\right) e^{-\phi}dv \\
	& -\int_{\mathbb{R}_{r_0}^m \cap (B(R)\backslash B\left(R_2\right))} m^{\frac{c}{2}-1}\left\|u^* h\right\|^{p-2} g^{j k} g^l \frac{\partial \tilde{u}_\alpha}{\partial x_i} \frac{\partial \varphi_1\left(\frac{r(x)}{R_1}\right)}{\partial x_j} \frac{\partial \tilde{u}_\beta}{\partial x_k} \frac{\partial \tilde{u}_\beta}{\partial x_l} \hat{u}_{\alpha} e^{-\phi} f^{m-4} d v_{g_0} \\
	& -\int_{\mathbb{R}_{r_0}^m \cap (B(R)\backslash B\left(R_2\right))} \frac{m-2 p}{p^2}|d u|^{2 \mathrm{p}-2} g^{i j} \frac{\partial \tilde{u}_{\alpha}}{\partial x_i} \frac{\partial \varphi_1\left(\frac{r[x]}{R_1}\right)}{\partial x_j} \hat{u}_{\alpha} e^{-\phi} f^{m-2} d v_{g_0} .
	&
\end{aligned}
$$
In \eqref{ccc}, we let $ s\to 0, $ noticing that $ \tilde{u} \to u, \hat{u}\to \bar{u} $ we get (cf.\cite{132132132} )

\begin{equation}\label{eq1}
	\begin{aligned}
		& \int_{        \mathbb{R}_{+}^m                                                                  \cap (B(R)\backslash B\left(R_2\right))} m^{\frac{1}{2}-1}\left\|u^* h\right\|^{p-2} g^{i k} g^{j i} \frac{\partial u_\alpha}{\partial x_i} \frac{\partial \bar{u}_\alpha}{\partial x_j} \frac{\partial u_\beta}{\partial x_k} \frac{\partial u_\beta}{\partial x_l} e^{-\phi} f^{m-4} d v_{g_0} \\
		& +\int_{      \mathbb{R}_{+}^m                                                                  \cap (B(R) \backslash B\left(R_2\right))} \frac{m-2 p}{p^2}|d u|^{2 p-2} g^{i j} \frac{\partial u_\alpha}{\partial x_i} \frac{\partial \bar{u}_\alpha}{\partial x_j} e^{-\phi} f^{m-2} d v_{g_0} \\
		& -\int_{      \mathbb{R}_{+}^m                                                                  \cap  \left( B(R) \backslash B (R_2)\right)  } \frac{1}{\epsilon^n}\left(1-|u|^2\right) u_\alpha \bar{u}_\alpha e^{-\phi}d v_g+D^*\left(R_1\right) \\
		=&\int_{\partial B(R) \cap       \mathbb{R}_{+}^m                                                                  } m^{\frac{p}{2}-1}\left\|u^* h\right\|^{p-2} g_0^{i k} g_0^{jl} \frac{\partial u_\alpha}{\partial x_i} \frac{\partial r(x)}{\partial x_j} \frac{\partial u_\beta}{\partial x_k} \frac{\partial u_\beta}{\partial x_i} \bar{u}_{\alpha} e^{-\phi} f^{m-4}d s \\
		& +\int_{\partial B(R) \cap \mathbb{R}_{r_0}^m} \frac{m-2 p}{p^2}|d u|^{2 p-2} g_0^{i j} \frac{\partial u_\alpha}{\partial x_i} \frac{\partial r(x)}{\partial x_j} \bar{u}_{\alpha} e^{-\phi} f^{m-2} d s, \\
	\end{aligned}
\end{equation}
Here  $\bar{u}_\alpha=\frac{u_{\alpha}^2-c_\alpha^2}{u_\alpha}, \alpha \in\{1,2, \cdots, n\}$, $ D^*\left(R_1\right) =\lim\limits_{s\to 0} D(R_1) $. Notice that  the notations $ \tilde{u} $ in \cite[Theorem 3.3]{132132132} is $ \bar{u} $ here. Hence by (27) and (28) in \cite{132132132}, 

$$
	\frac{\partial \bar{u}_a}{\partial x_j}=\left(1+\frac{c_\alpha^2}{u_\alpha^2}\right) \frac{\partial u_\alpha}{\partial x_j}, \quad \quad \sum_{a=1}^n u_\alpha \bar{u}_\alpha=-\left(1-\sum_{\alpha=1}^n u_\alpha^2\right)
$$
The LHS of (\ref{eq1}) is bounded below by (cf. \cite{132132132} ) 
$$
\begin{aligned}
	 &\int_{      \mathbb{R}_{+}^m                                                                  \cap (B(R)\backslash B\left(R_2\right))} m^{\frac{p}{2}-1} \| u^*h \|^{p-2} g_0^{i k} g_0^{j l} \frac{\partial u_\alpha}{\partial x_i} \frac{\partial u_\alpha}{\partial x_j} \frac{\partial u_\beta}{\partial x_k} \frac{\partial u_\beta}{\partial x_l}\left(1+\frac{c_\alpha^2}{u_\alpha^2}\right) e^{-\phi} f^{m-4}d v_{g_0} \\
	& +\int_{      \mathbb{R}_{+}^m                                                                  \cap (B(R)\backslash B\left(R_2\right))} \frac{m-2 p}{p^2}|d u|^{2 p-2} g_0^{i j} \frac{\partial u_\alpha}{\partial x_i} \frac{\partial u_\alpha}{\partial x_j}\left(1+\frac{c_\alpha^2}{u_\alpha^2}\right)e^{-\phi} f^{m-2} d v_{g_0} \\
	& +\int_{      \mathbb{R}_{+}^m                                                                  \cap (B(R)\backslash B\left(R_2\right))} \frac{1}{\epsilon^n}\left(1-|u|^2\right)^2e^{-\phi} d v_g+D^*\left(R_1\right) \\
	 \geq& \int_{      \mathbb{R}_{+}^m                                                                  \cap (B(R)\backslash B\left(R_2\right))}\left[m^{\frac{p}{2}-1}\left\|u^* h\right\|^p+\frac{m-2 p}{p^2}|d u|^{2 p}+\frac{1}{\varepsilon^n}\left(1-|u|^2\right)^2\right]  e^{-\phi} d v_g+D^*\left(R_1\right) .
\end{aligned}
$$
By (\ref{eq1})， we have 

\begin{equation}\label{rrr}
	\begin{aligned}
		& \int_{\partial B(R) \cap \mathbb{R}_{+}^m} m^{\frac{p}{2}-1}\left\|u^* h\right\|^{p-2} g_0^{i k} g_0^{j l} \frac{\partial u_\alpha}{\partial x_i} \frac{\partial r(x)}{\partial x_j} \frac{\partial u_\beta}{\partial x_k} \frac{\partial u_\beta}{\partial x_l} {u}_\alpha  e^{-\phi} f^{m-4} dv_{g_0} \\
	& +\int_{\partial B(R) \cap \mathbb{R}_{+}^m} \frac{m-2 p}{p^2}|d u|^{2 p-2} g_0^{i j} \frac{\partial u_\alpha}{\partial x_i} \frac{\partial r(x)}{\partial x_j} {u}_\alpha e^{-\phi} f^{m-2} dv_{g_0} \\
	 \geq &\int_{      \mathbb{R}_{+}^m                                                                  \cap (B(R)\backslash B\left(R_2\right))}\left[m^{\frac{p}{2}-1}\left\|u^* h\right\|^p+\frac{m-2 p}{p^2}|d u|^{2 p}+\frac{1}{\epsilon^n}\left(1-|u|^2\right)^2\right]e^{-\phi} dv_g +D^*\left(R_1\right) \text {. } \\
	&	
	\end{aligned}
\end{equation}
%覀估计式 $(\ref{rrr})$ 的 left hand side 。 选取点 $p \in \partial B(R)$, 因为
%$$
%\begin{aligned}
%	& \sum_{i, j, k, i=1}^m \sum_{\alpha, j=1}^n m^{\frac{2}{2}-1} \| u^{+} h^p{ }^{p-2} g^{i k} g^{j l} \frac{\partial u_\alpha}{\partial x_i} \frac{\partial r(x)}{\partial x_j} \frac{\partial u_\beta}{\partial x_k} \frac{\partial u_\beta}{\partial x_1} \bar{u}_a \\
%	+ & \sum_{i, j=1}^m \sum_{a=1}^n \frac{m-2 p}{p^2}|d u|^{2 p-2} g^{i j} \frac{\partial u_\alpha}{\partial x_i} \frac{\partial r(x)}{\partial x_j} {u}_\alpha
%\end{aligned}
%$$
Modifying the argument in  \cite{132132132}, at the point $p$, we have 
$$
\begin{aligned}
	& \sum_{i, j, k,l=1}^m \sum_{\alpha, \beta=1}^n m^{\frac{p}{2}-1}\left\|u^* h\right\|^{p-2} g_0^{i k} g_0^{j i} \frac{\partial u_\alpha}{\partial x_i} \frac{\partial r(x)}{\partial x_j} \frac{\partial u_\beta}{\partial x_k} \frac{\partial u_\beta}{\partial x_l} \bar{u}_\alpha e^{-\phi} f^{m-4}  \\
	& +\sum_{i,j=1}^m \sum_{\alpha=1}^n \frac{m-2 p}{p^2}|d u|^{2 p-2} g_0^{i j} \frac{\partial u_\alpha}{\partial x_i} \frac{\partial r(x)}{\partial x_j} \bar{u}_\alpha e^{-\phi} f^{m-2}\\
	& \leq C(m, p)\left[m^{\frac{p}{2}-1}\left\|u^* h\right\|^p+\frac{m-2 p}{p^2}|d u|^{2 p}\right]^{\frac{2 p-1}{2 p}}\left[\sum_{\alpha=1}^n\left({u}_\alpha\right)^2\right]^{\frac{1}{2}} e^{-\phi} f^{m}, \\
	&
\end{aligned}
$$
where $C(m, p)=2^{\frac{1}{p}} \max \left\{m^{\frac{p-1}{2 p}},\left(\frac{m-2 p}{p^2}\right)^{\frac{1}{2 p}}\right\},|\nabla r|=1$. Since $f(x)=x^{\frac{2p-1}{2 p}}, x \geq 0$  is concave，integrate this inequality over $\partial B(R) \cap \mathbb{R}_{r_0}^m$, we have 
\begin{equation}\label{dfg}
	\begin{aligned}
			& \int_{\partial B(R) \cap \mathbb{R}_{+}^m} m^{\frac{p}{2}-1} \| u^* h\|^{p-2} g_0^{i k} g_0^{jl} \frac{\partial u_\alpha}{\partial x_i} \frac{\partial r(x)}{\partial x_j} \frac{\partial u_\beta}{\partial x_k} \frac{\partial u_j}{\partial x_l} {u}_\alpha e^{-\phi} f^{m-4} dS_{g_0} \\
		& +\int_{\partial B(R) \cap \mathbb{R}_{+}^m} \frac{m-2 p}{p^2}|d u|^{2 p-2} g_0^{i j} \frac{\partial u_\alpha}{\partial x_i} \frac{\partial r(x)}{\partial x_j} {u}_\alpha e^{-\phi} f^{m-2} dS_{g_0}\\
		 \leq &  C(m, p)\left[\int_{\partial B(R) \cap \mathbb{R}_{+}^m}\left[m^{\frac{p}{2}-1}\left\|u^* h\right\|^p+\frac{m-2 p}{p^2}|d u|^{2 p}+\frac{1}{\varepsilon^n}\left(1-|u|^2\right)^2\right]e^{-\phi} f^{m}  d S_{g_0}\right]^{\frac{2 p-1}{2 p}} \\
		& \times\left[\int_{\partial B(R) \cap \mathbb{R}_{+}^m}\left(\sum_{a=1}^n {u}_\alpha^2\right)^p e^{-\phi} f^{m}d S_{g_0}\right]^{\frac{1}{2 p}} .	
	\end{aligned}
\end{equation}
By (\ref{rrr}) and \eqref{dfg}, we have 
\begin{equation}\label{gr}
	\begin{aligned}
		& \int_{      \mathbb{R}_{+}^m                                                                  \cap (B(R)\backslash B\left(R_2\right))}\left[m^{\frac{p}{2}-1}\left\|u^* h\right\|^p+\frac{m-2 p}{p^2}|d u|^{2 p}+\frac{1}{\epsilon^n}\left(1-|u|^2\right)^2\right] e^{-\phi}d v_g+D\left(R_1\right) \\
		\leq & C(m, p)\left[\int_{\partial B(R) \cap \mathbb{R}_{r_0}^m}\left( m^{\frac{p}{2}-1}\left\|u^* h\right\|^p+\frac{m-2 p}{p^2}|d u|^{2 p}+\frac{1}{\varepsilon^n}\left(1-|u|^2\right)^2\right) e^{-\phi} f^{m}  d S_{g_0}\right]^{\frac{2 p-1}{2 p}} \\
		& \times\left[\int_{\partial B(R) \cap \mathbb{R}_{r_0}^m}\left(\sum_{a=1}^n u_\alpha^2\right)^p e^{-\phi} f^{m}  d S_{g_0}\right]^{\frac{1}{p}},
	\end{aligned}
\end{equation}
Set
$$
Z(R)=\int_{      \mathbb{R}_{+}^m                                                                  \cap (B(R)\backslash B\left(R_2\right))}\left[ m^{\frac{p}{2}-1}\left\|u^* h\right\|^p+\frac{m-2 p}{p^2}|d u|^{2 p}+\frac{1}{\varepsilon^n}\left(1-|u|^2\right)^2\right] e^{-\phi}d v_g+D^*\left(R_1\right)_{,}
$$

$$
\frac{d}{d R} Z(R)=\int_{\partial B(R) \cap \mathbb{R}_{+}^m}\left[m^{\frac{p}{2}-1}\left\|u^2 h\right\|^p+\frac{m-2 p}{p^2}|d u|^{2p}+\frac{1}{\epsilon^n}\left(1-|u|^2\right)^2\right]e^{-\phi} dS_g
$$
The reamined proof is almost the same as \cite{132132132}. By \eqref{gr}, we have 
$$
Z(R) \leq C(m, p)\left[\frac{d}{d R} Z(R)\right]^{\frac{2 p-1}{2 p}}\left[\int_{\partial B(R)}\left(\sum_{\alpha=1}^m {u}_\alpha^2\right)^p e^{-\phi} f^{m}  d S_{g_0}\right]^{\frac{1}{2 p}}.
$$
On the other hand ，
$$
Z(R)-D^*\left(R_1\right)=\int_{\mathbb{R}_{+}^m \cap \left( B(R) \backslash B\left(R_2\right)\right) }\left[m^{\frac{p}{2}-1} \| u^* h\|^p+\frac{m-2 p}{p^2}|d u|^{2 p}+\frac{1}{e^n}\left(1-|u|^2\right)^2\right]e^{-\phi} d v_g.
$$
Since $\Phi_{s, p, t}(u)$ is  undoubded ,  there exists $R_3>R_2$ ,such that  $R>R_3, Z(R)>0$. i.e.
$$
M(R)=\left[\int_{\mathbb{R}_{+}^m \cap \partial B(R)}\left(\sum_{n=1}^m u_\alpha^2\right)^p e^{-\phi} f^{m}  d S_{g_0} \right]^{\frac{1}{2p-1}}.
$$
We get 
$$
\frac{\frac{d Z(R)}{d R}}{Z^{\frac{2p}{2 p-1}}(R)} \geq \frac{1}{C(m, p)^{\frac{2 p}{2 p-1}} M(R)}
.$$
For $R_4>R>R_3$,
$$
\int_R^{R_4} \frac{\frac{d Z(r)}{d r}}{Z^{\frac{2 p}{2p-1}}(r)} d r \geq \frac{1}{C(m, p)^{\frac{2 p}{2p-1}} } \int_R^{R_4} \frac{1}{M(r)} d r.
$$
Let $R_4 \rightarrow \infty$  and noticing that  $Z(R)>0$, we get
\begin{equation}\label{qw}
	\begin{aligned}
		Z(R) \leq C(m, p)\left[\frac{1}{\int_R^{\infty} \frac{1}{M(r)} d r}\right]^{2 p-1}.
	\end{aligned}
\end{equation}
However,
$$
M(R) \leq \eta^{\frac{1}{2p-1}}(R)\left(\int_{\partial B(R)} e^{-\phi} f^{m}  d S_{g_0}\right)^{\frac{1}{2 p-1}}.$$

As in \cite{132132132}, we choose function $ \eta(R) $ such that 

(i) On $\left(R_3, \infty\right)$, $\eta(R)$ is decreasing  and when  $R \rightarrow \infty$, $m(R) \rightarrow 0$;

(ii) $\eta(R) \geq \max\limits_{r(x)=R}\left(\sum_{\alpha=1}^m {u}_\alpha^2\right)^2$.

It follows that 
$$
\int_R^{\infty} \frac{1}{M(r)} d r \geq \frac{1}{\eta^{\frac{1}{2 p-1}}(R)} \int_R^{\infty} \frac{1}{\left(\int_{\partial B(R)} e^{-\phi} f^{m}  d S_{g_0}\right)^{\frac{1}{2 p-1}}} d r \geq \frac{1}{\eta^{\frac{1}{2p-1}}(R)} R^{-\frac{c}{2p-1}},
$$
by (\ref{qw}), we have  
$$ Z(R) \leq C(m, p) \eta(R) R^{c}, R>R_3, $$
thus by the definition of $ Z(R) $ and $ E_{\Phi_{S, p, \epsilon}}^R(u),$ we get
$$
E_{\Phi_{S, p, \epsilon},\phi}^R(u) \leq C(m, p)\left(\frac{\eta(R)}{2 p}+\frac{c(u)}{R^{C_0}}\right) R^{c}.
$$	
This is an contradiction.

\end{proof}

Similarly, by the same argument as in Theorem \ref{369}, we have 
	\begin{thm}
	Let $u:\left(\mathbf{R}_{+}^{m}, f^2g_{0}, e^{-\phi} \mathrm{d} v_{g}\right) \rightarrow\left(N^{n}, h\right)$ be   $\phi$-$ \Phi_{S, p, \epsilon} $ harmonic map, and  the conditions \eqref{erf} and \eqref{er5} hold. For any $p \in N^{n}$, if there is an (nonempty) open neighbourhood $U_{p} \subset N^{n}$, such that the family of open sets $\left\{U_{p} | p \in N^{n}\right\}$ has the following property: for some $p \in N^{n}, u(x) \in U_{p}$ as $r(x) \rightarrow \infty$, then $u$ is a constant map.	
	%		 satisfying $F^{\prime}\left(\frac{|d u|^{2}}{2}\right)<+\infty, F^{\prime}\left(\frac{\left|u^{*} h\right|^{2}}{4}\right)<+\infty$ and the $F$-lower degree $l_{F}>0$ 

\end{thm}

  By modifying the proof of  Theorem 4.1 in \cite{cao2022liouville}, it is not hard to see that 

	\begin{thm} Let $ (\mathbf{R}_{+}^{m}, g=f^2g_0,e^{-\phi}\mathrm{d} v_g) $ be a complete metric measure space with a pole $  x_0.$  Let   $ (N^n, h) $ be a Riemannian manifold. 
	Let  $ u: (\mathbf{R}_{+}^{m},g=f^2 g_0,e^{-\phi}\d v_g)\to (N,h)  $ be  Ginzburg-Landau type $\phi$-$F$ harmonic map coupled with $\phi$-$F$ symphonic map (cf. \cite{cao2022liouville}), satisfying $F^{\prime}\left(\frac{|d u|^{2}}{2}\right)$ $<+\infty,F^{\prime}\left(\frac{|u^{*}h|^{2}}{4}\right)$ $<+\infty$  and the $F$-lower degree $l_{F}>0$.
	%		
	%		 the critical point of 
	%		\begin{equation}\label{et}
		%			\begin{split}
			%				E(u)=\int_{M}\left( F(\frac{|du |^2}{2})+F(\frac{|u^{*}h|^2}{4})+\frac{1}{4\epsilon^n}(1-|u|^2)^2\right) e^{-\phi}d\nu_{g},
			%			\end{split}
		%			\end{equation} 	where $ F, G\in C^2(M), \phi >0. $ 
	If the following condition holds, 
	\begin{equation}\label{ball2}
		\begin{split}
			\int_{R}^{\infty} \frac{1}{\bigg[\left( \int_{\partial B(R)} e^{-\phi} f^{m-2}\d v_{g_0}\right) ^{\frac{1}{2}} \bigg]^{\frac{4}{3}}}\d r\geq R^{-\frac{\sigma}{3}}.
		\end{split}
	\end{equation}
and 
\begin{equation}\label{er4}
	\begin{aligned}
		(\sigma+4d_F-m+\frac{\partial \phi}{\partial x_i}x_i) f(x) \leq (d_F-\frac{m}{2}) \left(2f \frac{\partial f}{\partial x_{i}} \cdot x_{i}\right) , 
	\end{aligned}
\end{equation}
We assume that  $ vol_g(B(R)) = o(R^\sigma), vol_g(\partial B(R)) \gg \frac{1}{4} $ 
$$ \partial_r\phi \geq 0 , f\geq 1,r\frac{\partial \log f }{\partial r}\geq 0, d_F \leq \frac{m}{4}. $$		
Furtherly, $ u (x) \to p_0 $ as $ |x|\to \infty $, then $ u $  is a constant.
\end{thm}
\begin{rem}
	The conditions in this theorem is too restrictive, thus it needs to be improved. 
\end{rem}

	%If for sufficiently large R,  \begin{equation*}
		%			\begin{split}
			%				\bigg(\int_{R}^{\infty} \frac{1}{\text{Vol}^{\frac{1}{2p-1}}(\partial B_r)} dr\bigg)^{-1} \leq CR^{\frac{\sigma}{2p-1}},
			%			\end{split}
		%		\end{equation*}
%	Assume that there exists two
%	positive functions $h_1(r)$ and $h_2(r)$ such that
%	$$h_1(r)[g - dr \otimes dr] \leq \operatorname{Hess}(r)\leq  h_2(r)[g - dr \otimes dr].$$
%	Suppose that 
%	\begin{equation*}
%		\begin{split}
%			1+(m-1)rh_1(r)-4rd_Fph_2(r)\geq \sigma>1.  \quad 
%		\end{split}
%	\end{equation*}

%
%	\bibliographystyle{acm}	
%	\bibliography{mybib2022}  

\end{document}